\documentclass[final]{siamltex}
\usepackage{graphicx}   
\usepackage{multirow}
\usepackage{amsmath}
\usepackage{amsfonts}
\usepackage{latexsym}
\usepackage[notcite,notref]{showkeys}
\usepackage{amsmath}    
\usepackage{amssymb, algorithm, algorithmic}
\usepackage{epstopdf}
\usepackage{epsfig}

\newcommand{\bbr}{\mathbb{R}}

\newtheorem{assumption}{Assumption}
\newtheorem{remark}[theorem]{Remark}
\usepackage[notcite,notref]{showkeys}

\newcommand{\beq}{\begin{equation}}
\newcommand{\eeq}{\end{equation}}
\newcommand{\beqa}{\begin{eqnarray}}
\newcommand{\eeqa}{\end{eqnarray}}
\newcommand{\beqas}{\begin{eqnarray*}}
\newcommand{\eeqas}{\end{eqnarray*}}
\newcommand{\bi}{\begin{itemize}}
\newcommand{\ei}{\end{itemize}}
\newcommand{\ba}{\begin{array}}
\newcommand{\ea}{\end{array}}

\newcommand{\nn}{\nonumber}

\def\eqnok#1{(\ref{#1})}
\def\zi{{z^{(i)}}}
\def\xtoi{{x^{(i)}}}

\def\argmin{{\rm argmin}}

\def\w{\omega}
\def\cD{{\cal D}} 

\newcommand{\bbe}{\mathbb{E}}
\def\prob{\mathop{\rm Prob}}
\def\Prob{{\hbox{\rm Prob}}}
\def\CP{{\cal P}}
\def\PG{{{\cal G}}}

\def\vgap{\vspace*{.1in}}

\title{Stochastic Block Mirror Descent Methods for \\
Nonsmooth and Stochastic Optimization
\thanks{September, 2013.
This research was partially supported by NSF
grants CMMI-1000347, CMMI-1254446, DMS-1319050, and ONR grant N00014-13-1-0036.
}}

\author{
 Cong D. Dang
    \thanks{Department of Industrial and Systems
    Engineering, University of Florida, Gainesville, FL, 32611.
    (email: {\tt congdd@ufl.edu}). }
    \and
     Guanghui Lan
    \thanks{Department of Industrial and Systems
    Engineering, University of Florida, Gainesville, FL, 32611.
    (email: {\tt glan@ise.ufl.edu}).}
}

\begin{document}

\maketitle
\begin{abstract}
In this paper, we present a new stochastic algorithm,
namely the stochastic block mirror descent (SBMD) method for
solving large-scale nonsmooth and stochastic optimization problems.
The basic idea of this algorithm is to incorporate the block-coordinate decomposition
and an incremental block averaging scheme into the classic (stochastic)
mirror-descent method, in order to significantly reduce the cost per
iteration of the latter algorithm. We establish the rate of convergence of the SBMD method
along with its associated large-deviation results for solving
general nonsmooth and stochastic optimization problems.
We also introduce different variants of this method and establish their
rate of convergence for solving strongly convex,
smooth, and composite optimization
problems, as well as certain nonconvex optimization problems.
To the best of our knowledge, all these developments related
to the SBMD methods are new in the stochastic optimization literature.
Moreover, some of our results also seem to be new for block coordinate
descent methods for deterministic optimization.

\vspace{.1in}


\vspace{.1in}

\end{abstract}

\setcounter{equation}{0}
\section{Introduction} \label{sec_intro}

The basic problem of interest in the paper is the stochastic
programming (SP) problem given by
\beq \label{sp}
f^* := \min_{x \in X} \{ f(x) := \bbe[F(x, \xi)]\}.
\eeq
Here $X \in \bbr^n$ is a closed convex set,
$\xi$ is a random variable with support $\Xi \subseteq \bbr^d$
and $F(\cdot, \xi): X \to \bbr$ is continuous for every $\xi \in \Xi$.
In addition, we assume that $X$ has a block structure, i.e.,
\beq \label{defX}
X = X_1  \times X_2  \times  \cdots  \times X_b,
\eeq
where $X_i\subseteq \bbr^{n_i}$, $i = 1, \ldots,b$, are closed convex sets with $n_1+n_2+\ldots+n_b=n$.

The last few years have seen a resurgence of interest in
the block coordinate descent (BCD) method for solving problems with $X$ given in the form of
\eqnok{defX}. In comparison with regular first-order methods,
each iteration of these methods updates only one block of
variables. 
In particular, if each block consists of only one variable (i.e., $n_i = 1, i =1, \ldots,b$),
then the BCD method becomes the simplest coordinate descent (CD) method.
Although simple, these methods are found to be effective in solving huge-scale
problems with $n$ as big as $10^8 - 10^{12}$
(see, e.g.,~\cite{Nest10-1,LevLew10-1,Nest12-1,Rich12-1,BeckTet13-1}),
and hence are very useful for dealing with high-dimensional problems, especially those
from large-scale data analysis applications.
While earlier studies on the BCD method were focused on their asymptotically convergence behaviour
(see, e.g.,~\cite{LuoTseng91-1,Tseng01-1} and also \cite{TsengYun09-1,Wright10-1}),
much recent effort has been directed to the complexity analysis of these types of
methods (see \cite{Nest10-1,LevLew10-1,ShaTew11,Rich12-1,BeckTet13-1}.
In particular, Nesterov~\cite{Nest10-1} was among the first (see also~Leventhal and Lewis~\cite{LevLew10-1},
and Shalev-Shwartz and Tewari~\cite{ShaTew11}) to analyze the iteration complexity of a randomized BCD method
for minimizing smooth convex functions. More recently, the BCD methods were further enhanced by
Richt\'{a}rik and Tak\'{a}\u{c}~\cite{Rich12-1}, Beck and Tetruashvili~\cite{BeckTet13-1}, Lu and Xiao~\cite{LuXiao13-1}, etc.
We refer to \cite{Rich12-1} for an excellent review on the earlier developments of BCD methods.

However, to the best of our knowledge, most current BCD methods
were designed for solving deterministic optimization problems.
One possible approach for solving problem \eqnok{sp}, based on existing BCD methods
and the sample average approximation (SAA)~\cite{ShDeRu09},
can be described as follows.
For a given set of i.i.d. samples (dataset) $\xi_k, k = 1, \ldots, N$, of $\xi$,
we first approximate $f(\cdot)$ in \eqnok{sp} by
$
\tilde{f}(x) := \frac{1}{N} \sum_{k=1}^N F(x,\xi_k)
$
and then apply the BCD methods to $\min_{x \in X} \tilde f(x)$.
Since $\xi_k$, $k =1, \ldots, N$, are fixed a priori,
by recursively updating the (sub)gradient of $\tilde{f}$ (see~\cite{Nest10-1,Nest12-1}),
the iteration cost of the BCD method can be considerably
smaller than that of the gradient descent methods.
However, the above SAA approach is also known for the following drawbacks:
a) the high memory requirement to store $\xi_k$, $k =1, \ldots, N$;
b) the high dependence (at least linear) of the iteration cost on the sample size $N$,
which can be expensive when dealing with large datasets; and c)
the difficulty to apply the approach to the on-line setting where one needs to update
the solution whenever a new piece of data $\xi_k$ is collected.

A different approach to solve problem \eqnok{sp} is called stochastic approximation (SA),
which was initially proposed by Robbins and Monro~\cite{RobMon51-1} in 1950s
for solving strongly convex SP problems (see also \cite{pol90,pol92}). The SA method has also attracted much interest
recently (see, e.g.,~\cite{NJLS09-1,Lan10-3,NiuRech11-1,GhaLan12-2a,GhaLan13-1,LinChePe11-1,GhaLan12,NO13-1,GhaLanZhang13-1,JNTV05-1,JRT08-1,Nest06-2,Xiao10-1}).
In particular, Nemirovski et. al. \cite{NJLS09-1} presented
a properly modified SA approach, namely, the mirror descent SA for solving general nonsmooth convex SP problems.
Lan~\cite{Lan10-3} introduced a unified optimal SA method for smooth, nonsmooth and stochastic optimization
(see also~\cite{GhaLan12-2a,GhaLan13-1} for a more general framework).
Ghadimi and Lan~\cite{GhaLan12} presented novel SA methods for nonconvex optimization (see also~\cite{GhaLanZhang13-1}).
Related methods, based on dual averaging, have been studied in \cite{JNTV05-1,JRT08-1,Nest06-2,Xiao10-1}.
Note that all these SA algorithms only need to access one single
$\xi_k$ at each iteration, and hence does not require much memory.
In addition, their iteration cost is independent of the sample size $N$.
However, since these algorithms need to update the whole vector $x$ at each iteration,
their iteration cost can strongly depend on $n$ unless the problem is very sparse (see, e.g., \cite{NiuRech11-1}).
In addition, it is unclear whether the SA methods can benefit from
the recursive updating as in the BCD methods, since the samples $\xi_k$ used in
different iterations are supposed to be independent.


Our main goal in this paper is to present a new class of stochastic methods,
referred to as the stochastic block mirror descent (SBMD) methods,
by incorporating the aforementioned block-coordinate decomposition
into the classic (stochastic) mirror descent method (\cite{nemyud:83,BeckTeb03-1,NJLS09-1}).
Our study has been mainly motivated by solving an important class of
SP problems with $F(x, \xi) = \psi(B x, \xi)$, where $B$ is a certain linear operator
and $\psi$ is a relatively simple function. These problems arise
from many machine learning applications, where $\psi$ is a loss function
and $B$ denotes a certain basis (or dictionary)
obtained by, e.g., metric learning (e.g.,~\cite{Xing02distancemetric}).
Each iteration of existing SA methods would require
${\cal O}(n^2)$ arithmetic operations to compute $B x$ and becomes prohibitive if $n$ exceeds $10^6$.
On the other hand, by using block-coordinate decomposition with $n_i =1$,
the iteration cost of the SBMD algorithms can be significantly reduced to ${\cal O} (n)$,
which can be further reduced if $B$ and $\xi_k$ are sparse (see Subsection~\ref{sec_alg}
for more discussions).
Our development has also been motivated by the situation when the bottleneck of the mirror descent
method exists in the projection (or prox-mapping) subproblems (see \eqnok{def_prox_mapping}).
In this case, we can also significantly reduce the iteration cost by
using the block-coordinate decomposition, since each iteration of the SBMD method requires
only one projection over $X_i$ for some $1 \le i \le b$,
while the mirror descent method needs to perform the projections over $X_i$ for all $1 \le i \le b$.

Our contribution in this paper mainly lies in the following aspects.
Firstly, we introduce the block decomposition into the classic
mirror descent method for solving general nonsmooth optimization problems.
Each iteration of this algorithm updates one block
of the search point  along a stochastic (sub)gradient $G_{i_k}(x_k, \xi_k)$.
Here, the index $i_k$ is randomly chosen and $G(x, \xi)$ is an unbiased estimator
of the subgradient of $f(\cdot)$, i.e.,
\beq \label{assump1}
\bbe[G(x, \xi)] = g(x) \in \partial f(x), \ \ \forall x \in X.
\eeq
In addition, in order to
compute the output of the algorithm, we introduce
an {\sl incremental block averaging} scheme, which updates only
one block of the weighted sum of the search points in each iteration.
We demonstrate that if $f(\cdot)$ is a general nonsmooth convex function, then the number of iterations
performed by the SBMD method to find an $\epsilon$-solution of \eqnok{sp}, i.e.,
a point $\bar x \in X$ such that (s.t.) $\bbe[f(\bar x) - f^*] \le \epsilon$,
can be bounded by ${\cal O}(b / \epsilon^2)$. Here the expectation
is taken w.r.t. the random elements $\{i_k\}$ and $\{\xi_k\}$.
In addition, if $f(\cdot)$ is strongly convex, then
the number of iterations performed by the SBMD method (with a different stepsize policy and averaging scheme)
to find an $\epsilon$-solution of \eqnok{sp} can be bounded by ${\cal O}(b/\epsilon)$.
We also derive the large-deviation results associated with
these rates of convergence for the SBMD algorithm.
Secondly, we consider a special class of convex stochastic
composite optimization problems given by
\beq \label{spc}
\phi^* := \min_{x \in X} \left\{\phi(x) := f(x) + \chi(x) \right\}.
\eeq
Here $\chi(\cdot)$ is a relatively simple convex function and $f(\cdot)$
defined in \eqnok{sp} is a smooth convex function with Lipschitz-continuous
gradients $g(\cdot)$. We show that, by properly modifying the SBMD method,
we can significantly improve the aforementioned complexity bounds in
terms of their dependence on the Lipschitz constants of $g(\cdot)$. We show that the complexity bounds
can be further improved if $f(\cdot)$ is strongly convex.
Thirdly, we generalize our study to a class of nonconvex
stochastic composite optimization problems in the form of \eqnok{spc}, but
with $f(\cdot)$ being possibly nonconvex. Instead of using
the aforementioned incremental block averaging, we
incorporate a certain randomization scheme
to compute the output of the algorithm. We also establish the complexity
of this algorithm to generate an approximate stationary
point for solving problem \eqnok{spc}.

While this paper focuses on stochastic optimization, it is worth noting that some
of our results also seem to be new in the literature for the BCD methods for deterministic
optimization. Firstly,
currently the only BCD-type methods for solving general nonsmooth CP problems
are based on the subgradient methods without involving averaging, e.g.,
those by Polak and a constrained version by Shor (see Nesterov~\cite{Nest12-1}).
Our development shows that it is possible to develop new BCD type methods involving different
averaging schemes for convex optimization.
Secondly, the large-deviation result for the BCD methods for general nonsmooth
problems and the ${\cal O}(b/\epsilon)$ complexity result for the BCD methods for
general nonsmooth strongly convex problems are  new in the literature. Thirdly, it appears to us that
the complexity for solving nonconvex optimization by the BCD methods has not been studied before in the literature.

This paper is organized as follows. After reviewing some notations in Section 1.1,
we present the basic SBMD algorithm
for general nonsmooth optimization and discuss its convergence
properties in Section 2. A variant of this algorithm for solving
convex stochastic composite optimization problems, along with its
complexity analysis are developed in Section 3. A generalization
of this algorithm for solving nonconvex stochastic composite
optimization is presented in Section 4. Finally some brief concluding
remarks are given in Section 5.

\subsection{Notation and terminology}
Let $\bbr^{n_i}$, $i = 1, \ldots, b$, be Euclidean spaces equipped with
inner product $\langle \cdot, \cdot \rangle$ and norm $\|\cdot\|_i$
($\|\cdot\|_{i,*}$ be the conjugate) such that $\sum_{i=1}^b n_i = n$.
Let $I_n$ be the identity matrix in $\bbr^n$ and $U_i\in \bbr^{n \times n_i}, i=1,2,\ldots,b,$ be the set of matrices satisfying
\[
(U_1, U_2, \ldots, U_b ) = I_n .
\]
For a given $x \in \bbr^n$, we denote its $i$-th block by $x^{(i)}=U_i^Tx$, $i=1,\ldots,b$.
Note that $$x= U_1 x^{(1)} + \ldots + U_b x^{(b)}.$$ Moreover, we define
$$\| x\|^2 = \| x^{(1)}\|_{1}^2+\ldots+\| x^{(b)}\|_{b}^2.$$
and denote its conjugate by
$\| y\|_*^2 = \| y^{(1)}\|_{1,*}^2+\ldots+\| y^{(b)}\|_{b,*}^2$.

Let $X$ be defined in \eqnok{defX} and $f: X \to \bbr$ be
a closed convex function. For any $x \in X$, let $G(x, \xi)$ be a stochastic subgradient of $f(\cdot)$ such that
\eqnok{assump1} holds. We denote the partial stochastic subgradient of $f(\cdot)$ by
$G_i(x,\xi )=U_i^TG(x,\xi), \; i=1,2,...,b.$


\setcounter{equation}{0}
\section{The SBMD methods for nonsmooth convex optimization} \label{sec_nonsmooth}
In this section, we present the stochastic block coordinate descent
method for solving stochastic nonsmooth convex optimization problems
and discuss its convergence properties.
More specifically, we present the basic scheme of
the SBMD method in Subsection~\ref{sec_alg} and discuss its convergence
properties for solving general nonsmooth and
strongly convex nonsmooth problems in Subsections~\ref{sec_nonsmooth_ns} and \ref{sec_nonsmooth_s},
respectively.

Throughout this section, we assume that  $f(\cdot)$ in \eqnok{sp} is convex and its stochastic subgradients
satisfy, in addition to
\eqnok{assump1}, the following condition:
\beq \label{assump1.a}
\bbe[\| G_i(x,\xi)\|_{i,*}^2] \leq M_i^2, \; i=1,2,...,b.
\eeq
Clearly, by \eqnok{assump1} and \eqnok{assump1.a}, we have
\beq \label{bnd_G_i}
\|g_i(x)\|_{i,*}^2 = \|\bbe[G_i(x,\xi)]\|_{i,*}^2 \le
\bbe[\| G_i(x,\xi)\|_{i,*}^2] \le M_i^2, \; i=1,2,...,b,
\eeq
and
\beq \label{bnd_G}
\|g(x)\|_*^2 = \sum_{i=1}^b \|g_i(x)\|_{i,*}^2 \le \sum_{i=1}^b M_i^2.
\eeq

\subsection{The SBMD algorithm for nonsmooth problems} \label{sec_alg}
We present a general scheme of the SBMD algorithm,
based on Bregman's divergence, to solve stochastic convex optimization problems.

Recall that a function $\omega_i:X_i \rightarrow R$ is a distance generating function~\cite{NJLS09-1}
with modulus $\alpha_i$ with respect to $\| \cdot \| _i$, if $\omega$ is continuously
differentiable and strongly convex with parameter $\alpha_i$
with respect to $\| \cdot \|_i$. Without loss of generality, we assume throughout
the paper that $\alpha_i=1$ for any $i = 1, \ldots, b$. Therefore, we have
$$
\langle x  - z, \nabla \omega_i ( x ) - \nabla \omega_i { (z)} \rangle
\ge \| {x  - z} \|_i^2 \ \ \forall x ,z \in X_i .
$$
The prox-function associated with $\omega_i$ is given by
\beq \label{def_vi}
V_i (z,x) = \omega _i (x) - [\omega _i (z)
+ \langle {\omega '_i (z),x - z } \rangle ] \ \ \forall x, z \in X_i.
\eeq
The prox-function $V_i(\cdot,\cdot)$ is also called the Bregman's distance,
which was initially studied by Bregman \cite{Breg67} and later by many others
(see \cite{AuTe06-1,BBC03-1,Teb97-1} and references therein).
For a given $x \in X_i$ and $y \in \bbr^{n_i}$, we define the prox-mapping as
\beq \label{def_prox_mapping}
P_i (v,y,\gamma) = \arg \min\limits_{u \in X_i }
\langle {y,u} \rangle  + \frac{1}{\gamma}V_i (u,v).
\eeq
Suppose that the set $X_i$ is bounded, the distance generating function
$\w_i$ also gives rise to the  following characteristic entity that
will be used frequently in our convergence analysis:
\beq \label{def_D_i}
{\cal D}_i  \equiv {\cal D}_{\omega _i ,X_i }
:= \left( {\mathop {\max }\limits_{x \in X_i } \omega _i (x)
- \mathop {\min }\limits_{x \in X_i } \omega _i (x)} \right)^{\frac{1}{2}}.
\eeq
Let $x_1^{(i)} = \argmin_{x \in X_i} \w_i(x)$, $i =1, \ldots, b$. We can easily see that
for any $x \in X$,
\beq \label{bound_V_i}
V_i (x_1^{(i)} ,x^{(i)} ) = \w_i(x^{(i)})
- \w_i(x_1^{(i)}) - \langle \nabla
\w_i(x_1^{(i)}), x^{(i)}- x_1^{(i)}\rangle
\le \w_i(x^{(i)}) - \w_i(x_1^{(i)}) \le {\cal D}_i,
\eeq
which, in view of the strong convexity of $\w_i$, also implies that
$\|x_1^{(i)} - x^{(i)}\|_i^2 / 2 \le {\cal D}_i$. Therefore,
for any $x, y \in X$, we have
\begin{align}
\|x^{(i)} - y^{(i)}\|_i &\le \|x^{(i)} - x_1^{(i)}\|_i
+ \|x_1^{(i)} - y^{(i)}\|_i \le 2 \sqrt{2 \cD_i}, \label{boundX_i}\\
\|x - y\| &= \sqrt{\sum_{i=1}^b \|x^{(i)} - y^{(i)}\|_i^2}
\le 2 \sqrt{2\sum_{i=1}^b \cD_i}. \label{boundX}
\end{align}

\vgap

With the above definition of the prox-mapping,
we can formally describe the stochastic block mirror descent (SBMD) method
as follows.

\vgap

\begin{algorithm} [H]
    \caption{The Stochastic Block Mirror Descent (SBMD) Algorithm}
    \label{algSBMD}
    \begin{algorithmic}
\STATE Let $x_1 \in X$, stepsizes $\{\gamma_k\}_{k \ge 1}$,
weights $\{\theta_k\}_{k \ge 1}$, and
probabilities $p_i \in [0,1]$, $i = 1, \ldots, b$, s.t. $\sum_{i=1}^b p_i = 1$
be given. Set $s_1 = 0,$ and $u_i = 1$ for $i = 1, \ldots, b$.

\FOR {$k=1, \ldots,N$}

\STATE 1. Generate a random variable $i_k$ according to
\beq \label{eqn_prob}
\prob\left\{i_k = i\right\} = p_i, \ \  i = 1, \ldots, b.
\eeq

\STATE 2. Update $s^{(i)}_k$, $i = 1, \ldots, b$, by
\beq \label{eqn_sum}
s^{(i)}_{k+1}=
\left\{
\begin{array}{ll}
s_{k}^{(i)}  + x_{k}^{(i)} \sum_{j = u_{i_k}}^{k} \theta_j & i = i_k,\\
s^{(i)}_{k} & i \neq i_k,
\end{array}
\right.
\eeq

and then set $u_{i_k} = k+1$.
\STATE 3. Update
$x_k^{(i)}$, $i = 1, \ldots, b$, by
\beq \label{eqn_prox}
x^{(i)}_{k+1}=
\left\{
\begin{array}{ll}
P_{i}(x_{k}^{(i)}, G_{i_k}(x_{k},\xi_{k}), \gamma_k ) & i = i_k,\\
x^{(i)}_{k} & i \neq i_k.
\end{array}
\right.
\eeq
\ENDFOR

{\bf Output:}  Set $s_{N+1}^{(i)} = s_{N+1}^{(i)}  + x_{N}^{(i)} \sum_{j = u_{i}}^{N} \theta_j$, $i = 1, \ldots, b$,
and  ${\bar x}_N = s_{N+1} / \sum \limits_{k=1}^{N} \theta_k.$
    \end{algorithmic}
\end{algorithm}

We now add a few remarks about the SBMD algorithm stated above.
Firstly, each iteration of the SBMD method recursively updates
the search point $x_k$ based on the partial stochastic subgradient
$G_{i_k}(x_{k},\xi_{k})$. In addition, rather than
taking the average of $\{x_k\}$ in the end of algorithm as
in the mirror-descent method, we introduce an incremental
block averaging scheme to compute the output of the algorithm.
More specifically, we use a summation vector $s_k$ to
denote the weighted sum of $x_k$'s and
the index variables $u_i$, $i=1,\ldots,b$,
to record the latest iteration when the $i$-th block of
$s_k$ is updated. Then in \eqnok{eqn_sum},
we add up the $i_k$-th block of $s_k$ with $x_k \sum_{j=i_k}^k \theta_j $,
where $\sum_{j=i_k}^k \theta_j$
is often given by explicit formula and hence easy to compute.
It can be checked that by using this averaging scheme, we have
\beq \label{def_bar_x}
\bar x_N 
= \left(\sum_{k=1}^N \theta_k \right)^{-1}  \sum_{k=1}^N (\theta_k x_k).
\eeq

Secondly, observe that in addition to \eqnok{eqn_sum} and \eqnok{eqn_prox},
each iteration of the SBMD method involves the computation of
$G_{i_k}$. Whenever possible, we
should update $G_{i_k}$ recursively in order to reduce the iteration cost
of the SBMD algorithm. Consider an important class of SP problems with
the objective function
\[
f(x) = \bbe[\psi(B x - q, \xi)] + \chi(x),
\]
where $\psi(\cdot)$ and $\chi(\cdot)$ are relatively simple functions, $q \in \bbr^n$, and $B\in \bbr^{n \times n}$.
For the sake of simplicity, let us also assume that $n_1 = \ldots = n_b = 1$.
For example, in the well-known support vector machine (SVM) problem, we have $\psi(y) = \max\left\{
\langle y, \xi\rangle, 0\right\}$ and $\chi(x) = \|x\|_2^2/2$.
In order to compute the full vector $G(x_k, \xi_k)$,
we need ${\cal O}(n^2)$ arithmetic operations to compute the vector $B x_k - q$,
which majorizes other arithmetic operations if $\psi$ and $\chi$ are simple.
On the other hand, by recursively updating the vector $y_k = B x_k$
in the SBMD method, we can significantly reduce the iteration cost
from ${\cal O}(n^2)$ to ${\cal O}(n)$.
This bound can be further reduced if both  $\xi_k$ and $B$ are sparse (i.e., the vector $\xi_k$ and
each row vector of $B$ contain just a few nonzeros).
The above example can be generalized to the case when $B$ has $r \times b$ blocks
denoted by $B_{i,j} \in \bbr^{m_i \times n_j}$,
$1 \le i \le r$ and $1 \le j \le b$, and each block row $B_i = (B_{i,1}, \ldots, B_{i,b})$,
$i = 1, \ldots, r$, is block-sparse (see~\cite{Nest12-1} for some related discussion).

Thirdly, observe that the above SBMD method is conceptual only because we have
not yet specified the selection of the stepsizes $\{\gamma_k\}$,
the weights $\{\theta_k\}$, and the probabilities $\{p_i\}$.
We will specify these parameters after establishing some
basic convergence properties of this method.

\subsection{Convergence properties of SBMD for nonsmooth problems} \label{sec_nonsmooth_ns}
In this subsection, we discuss the main convergence properties of the SBMD
method for solving general nonsmooth convex problems.

\begin{theorem} \label{theorem_nonsmooth}
Let $\bar x_N$ be the output of the SBMD algorithm and suppose that
\beq \label{equal_theta_gamma}
\theta_k = \gamma_k, \ \ k = 1, \ldots, N.
\eeq
Then we have,
for any $N \ge 1$ and $x \in X$,
\beq \label{nonsmooth_result}
\bbe[f(\bar x_N) - f(x)]
\le \left( {\sum\limits_{k = 1}^N {\gamma _k } } \right)^{ - 1} \left [
\sum \limits_{i=1}^b p_i^{-1}V_i(x_1^{(i)}, x^{(i)})
+ \frac{1}{2}\sum\limits_{k = 1}^N {\gamma _k^2 } \sum\limits_{i = 1}^b {M_i^2}\right ] ,
\eeq
where the expectation is taken with respect to (w.r.t.) $\{i_k\}$ and $\{\xi_k\}$.
\end{theorem}

\begin{proof}
For simplicity, let us denote $V_i(z, x) \equiv V_i(z^{(i)}, x^{(i)})$, $g_{i_k} \equiv g^{(i_k)}(x_k)$ (c.f. \eqnok{assump1})
and $V(z,x)=\sum \limits_{i=1}^b p_i^{-1}V_i(z, x)$.
Also let us denote $\zeta_k = (i_k, \xi_k)$ and $\zeta_{[k]} = (\zeta_1, \ldots, \zeta_k)$.
By the optimality condition of \eqnok{def_prox_mapping} (e.g., Lemma 1 of \cite{NJLS09-1}) and the definition of
$x_k^{(i)}$ in \eqnok{eqn_prox}, we have
\[
V_{i_k } (x_{k + 1}, x ) \le {V_{i_k } (x_k ,x) +
  \gamma _k \left\langle {G_{i_k }(x_k, \xi_k), U_{i_k }^T (x - x_k )} \right\rangle
  + \frac{1}{2} \gamma _k^2
   \left\| {G_{i_k } (x_k ,\xi _k )} \right\|_{i_k ,*}^2  } .
\]
Using this observation, we have, for any $k \ge 1$ and $x \in X$,
\beq \label{eqn_nonsmooth}
\begin{array}{l}
 V(x_{k + 1} ,x) = \sum\limits_{i \ne i_k } {p_i^{ - 1} V_i (x_k,x)}
 + p_{i_k }^{ - 1} V_{i_k } (x_{k + 1}, x ) \\
  \le \sum\limits_{i \ne i_k } {p_i^{ - 1} V_i (x_k, x)}  +
  p_{i_k }^{ - 1} \left[ {V_{i_k } (x_k ,x) +
  \gamma _k \left\langle {G_{i_k }(x_k, \xi_k), U_{i_k }^T (x - x_k )} \right\rangle
  + \frac{1}{2} \gamma _k^2
   \left\| {G_{i_k } (x_k ,\xi _k )} \right\|_{i_k ,*}^2  } \right] \\
 = V(x_k, x) + \gamma _k p_{i_k }^{ - 1} \left\langle {U_{i_k } G_{i_k }(x_k, \xi_k), x - x_k } \right\rangle
  + \frac{1}{2}  \gamma _k^2
   p_{i_k }^{ - 1}   \left\| {G_{i_k } (x_k ,\xi _k )} \right\|_{i_k ,*}^2 \\
= V(x_k, x) + \gamma_k \langle g(x_k), x- x_k \rangle  + \gamma_k \delta_k  + \frac{1}{2} \gamma_k^2 \bar \delta_k,
 \end{array}
\eeq
where
\beq \label{def_delta}
\delta_k := \langle p_{i_k }^{ - 1}
{U_{i_k } G_{i_k }(x_k, \xi_k) - g(x_k), x - x_k } \rangle
\ \ \mbox{and} \ \
\bar \delta_k := p_{i_k }^{ - 1}   \left\| {G_{i_k } (x_k ,\xi _k )} \right\|_{i_k ,*}^2.
\eeq
It then follows from \eqnok{eqn_nonsmooth}
and the convexity of $f(\cdot)$ that, for any $k \ge 1$ and $x \in X$,
\[
\gamma_k [f(x_k) - f(x)] \le
\gamma_k \langle g(x_k), x_k - x\rangle \le V(x_k, x) - V(x_{k+1},x) +
\gamma_k \delta_k + \frac{1}{2}  \gamma_k^2 \bar \delta_k.
\]
By using the above inequalities, the convexity of $f(\cdot)$, and
the fact that
$\bar x_N = \sum_{k = 1}^N (\gamma _k x_k) / \sum_{k = 1}^N \gamma _k$ due to \eqnok{def_bar_x} and \eqnok{equal_theta_gamma},
we conclude that for any $N \ge 1$ and $x \in X$,
\beq \label{main_rec}
\begin{array}{lll}
f(\bar x_N) - f(x) &\le& \left(\sum\limits_{k = 1}^N \gamma _k\right)^{-1}\sum_{k=1}^N \gamma_k \left[f(x_k) - f(x)\right]\\
&\le& \left(\sum\limits_{k = 1}^N \gamma _k\right)^{-1} \left[V(x_1, x)
+ \sum\limits_{k=1}^N \left(\gamma_k \delta_k  + \, \frac{1}{2}  \gamma _k^2 \bar{\delta}_k \right)\right].
\end{array}
\eeq
Now, observe that by \eqnok{assump1} and \eqnok{eqn_prob},
\[
\begin{array}{lll}
\bbe_{\zeta_k}\left[p_{i_k }^{ - 1} \langle U_{i_k } G_{i_k }, x - x_k \rangle| \zeta_{[k-1]}\right]
&=&  \sum_{i=1}^b \bbe_{\xi_k} \left[\langle U_{i} G_i(x_k, \xi_k), x - x_k \rangle | \zeta_{[k-1]}\right] \\
&=& \sum_{i=1}^b \langle U_i g_i(x_k), x - x_k \rangle = \langle g(x_k), x - x_k \rangle,
\end{array}
\]
and hence that
\beq \label{zero-diff}
\bbe[\delta_k| \zeta_{k-1}] = 0.
\eeq
Also, by \eqnok{eqn_prob} and \eqnok{assump1.a},
\beq \label{bnd_M}
\bbe\left[p_{i_k }^{ - 1}  \left\| {G_{i_k } (x_k ,\xi _k )} \right\|_{i_k ,*}^2\right]
= \sum_{i=1}^b p_i p_i^{-1}  \left\| {G_{i } (x_k ,\xi _k )} \right\|_{i ,*}^2
\le \sum_{i=1}^b  M_i^2.
\eeq
Our result in \eqnok{nonsmooth_result} then immediately follows by taking expectation on
both sides of \eqnok{main_rec}, and using the previous observations in \eqnok{zero-diff} and \eqnok{bnd_M}.
\end{proof}

\vgap

Below we provide a few specialized convergence results for the SBMD algorithm after properly selecting
$\{p_i\}$, $\{\gamma_k\}$, and $\{\theta_k\}$.

\begin{corollary} \label{general2}
Suppose that $\{\theta_k\}$ in Algorithm~\ref{algSBMD} are set to \eqnok{equal_theta_gamma}.
\begin{itemize}
\item [a)] If $X$ is bounded, and $\{p_i\}$ and $\{\gamma_k\}$ are set to
\beq \label{def_nonsmooth_prob}
p_i = \frac{{\sqrt {\cD_i} }}{{\sum\limits_{i = 1}^b {\sqrt {\cD_i} } }}, \ \ i = 1, \ldots, b,  \ \ \mbox{and} \ \
\gamma_k = \gamma \equiv \frac{\sqrt{2} \sum \limits_{i=1}^b \sqrt{\cD_i}}{ \sqrt{N \sum\limits_{i=1}^b
M_i^2}}, \ \ k = 1, \ldots, N,
\eeq
then
\beq \label{nonsmooth_key_result}
\bbe[f(\bar x_N ) - f(x)] \le  \sqrt{\frac{2}{N}} \sum \limits_{i=1}^b \sqrt{  {\cal D}_i} \, \sqrt{\sum \limits_{i=1}^b  M_i^2} \ \
\forall x \in X.
\eeq
\item [b)] If $\{p_i\}$ and $\{\gamma_k\}$ are set to
\beq \label{def_nonsmooth_prob2}
p_i = \frac{1 }{b}, \ \ i = 1, \ldots, b,  \ \ \mbox{and} \ \
\gamma_k = \gamma \equiv \frac{\sqrt{2 b}  \tilde D}{ \sqrt{N \sum\limits_{i=1}^b   M_i^2}}, \ \ k = 1, \ldots, N,
\eeq
for some $\tilde D > 0$,
then
\beq \label{nonsmooth_key_result_u}
\bbe[f(\bar x_N ) - f(x)] \le  \sqrt{\sum \limits_{i=1}^b  M_i^2}
\left( \frac{\sum_{i=1}^b V_i (x_1^{(i)} ,x_*^{(i)} )}{\tilde D} +\tilde D \right) \frac{\sqrt{b}}{\sqrt{2N}} \ \
\forall x \in X.
\eeq
\end{itemize}
\end{corollary}

\begin{proof} We show part a) only, since part b) can be proved similarly.
Note that by
\eqnok{bound_V_i} and \eqnok{def_nonsmooth_prob}, we have
\[
\sum \limits_{i=1}^b p_i^{-1}V_i(x_1^{(i)}, x^{(i)}) \le \sum \limits_{i=1}^b p_i^{-1} \cD_i
\le
 \left( \sum\limits_{i = 1}^b \sqrt {\cD_i}  \right)^2.
\]
Using this observation, \eqnok{nonsmooth_result}, and \eqnok{def_nonsmooth_prob}, we have
\[
\bbe[f(\bar x_N ) - f(x_*)] \le (N \gamma)^{-1} \left[ \left( \sum\limits_{i = 1}^b \sqrt {\cD_i}  \right)^2
+ \frac{N \gamma^2}{2} \sum_{i=1}^b M_i^2\right]
 = \sqrt{\frac{2}{N}} \sum \limits_{i=1}^b \sqrt{  \cD_i} \, \sqrt{\sum \limits_{i=1}^b  M_i^2} .
\]
\end{proof}

\vgap

A few remarks about the results obtained in Theorem~\ref{theorem_nonsmooth}
and Corollary~\ref{general2} are in place.
First, the parameter setting in \eqnok{def_nonsmooth_prob} only works for
the case when $X$ is bounded, while the one in \eqnok{def_nonsmooth_prob2}
also applies to the case when $X$ is unbounded or when the bounds
$\cD_i$, $i =1, \ldots, b$, are not available.
It can be easily seen that the optimal choice of $\tilde D$ in \eqnok{nonsmooth_key_result_u}
would be $\sqrt{\sum_{i=1}^b V_i (x_1^{(i)} ,x_*^{(i)})}$.
In this case, \eqnok{nonsmooth_key_result_u} reduces
to
\beq \label{nonsmooth_bnd_simple}
\bbe[f(\bar x_N ) - f(x)] \le  \sqrt{2 \sum \limits_{i=1}^b  M_i^2}  \sqrt{\sum_{i=1}^b V_i (x_1^{(i)} ,x_*^{(i)} )}
\frac{\sqrt{b}}{\sqrt{N}} \le \sqrt{2 \sum \limits_{i=1}^b  M_i^2}  \sqrt{\sum_{i=1}^b {\cal D}_i}
\frac{\sqrt{b}}{\sqrt{N}},
\eeq
where the second inequality follows from \eqnok{bound_V_i}.
It is interesting to note the difference
between the above bound and \eqnok{nonsmooth_key_result}. Specifically,
the bound obtained in \eqnok{nonsmooth_key_result} by using a non-uniform
distribution $\{p_i\}$ always minorizes the one in \eqnok{nonsmooth_bnd_simple}
by the Cauchy-Schwartz inequality.

Second, observe that in view of \eqnok{nonsmooth_key_result},
the total number of iterations required by the SBMD method to find
an $\epsilon$-solution of \eqnok{sp} can be bounded by
\beq \label{iter_SBMD}
2 \left( \sum_{i=1}^b \sqrt{\cD_i} \right)^2 \left(\sum_{i=1}^b M_i^2\right)
\frac{1}{\epsilon^2}.
\eeq
Also note that the iteration complexity of the mirror-descent SA algorithm employed with
the same $\w_i(\cdot)$, $i = 1, \ldots, b$, is given by
\beq \label{iter_MDSA}
2 \sum \limits_{i=1}^b {\cal D}_i \, \left(\sum \limits_{i=1}^b  M_i^2\right) \, \,
\frac{1}{\epsilon^2}.
\eeq
Clearly, the bound in \eqnok{iter_SBMD} can be larger, up to
a factor of $b$, than the one in \eqnok{iter_MDSA}. Therefore, the total arithmetic cost
of the SBMD algorithm will be comparable to or smaller than that of the mirror descent SA,
if its iteration cost is smaller than that of the latter algorithm by a factor of ${\cal O} (b)$.

Third, in Corollary~\ref{general2} we have used a constant stepsize policy
where $\gamma_1 = \ldots = \gamma_N$. However, it should be noted that
variable stepsize policies, e.g., those similar to \cite{NJLS09-1}, can also be used
in the SBMD method.

\subsection{Convergence properties of SBMD for nonsmooth strongly convex problems} \label{sec_nonsmooth_s}

In this subsection, we assume that the objective function $f(\cdot)$ in \eqnok{sp} is strongly convex, i.e., $\exists$ $\mu >0$ s.t.
\beq \label{def_strong_convexity}
f(y) \ge f(x) + \langle  g(x), y-x\rangle +
\frac{\mu}{2} \| y-x \|^2 \ \ \forall \, x,y \in X.
\eeq
In order to establish the convergence of the SBMD algorithm for solving strongly convex problems,
we need to assume that the prox-functions
$V_i(\cdot,\cdot)$, $i =1, \ldots, b$, satisfy a quadratic growth condition (e.g.,~\cite{jn10-1,GhaLan12-2a,GhaLan13-1}):
\beq \label{q-growth}
V_i(\zi,\xtoi) \le \frac{Q}{2}\|\zi - \xtoi\|_i^2 \ \ \forall \zi, \xtoi \in X_i,
\eeq
for some $Q > 0$. In addition, we need to assume that the probability distribution of $i_k$ is uniform, i.e.,
\beq \label{uniform}
p_1 = p_2 = \ldots = p_b = \frac{1}{b}.
\eeq

Before proving the convergence of the SBMD algorithm
for solving strongly convex problems,
we first state a simple technical result obtained
by slightly modifying Lemma 3 of \cite{Lan13-2}.

\begin{lemma} \label{tech_result_sum} Let $a_k\in (0,1]$, $k = 1, 2, \ldots$, be given.
Also let us denote
\beq \label{def_Gamma0}
A_{k}:= \left\{
\begin{array}{ll}
1& k = 1\\
(1-a_k) \, A_{k-1}& k \ge 2.
\end{array}
\right.
\eeq
Suppose that $A_k > 0$ for all $k \ge 2$ and that the sequence $\{\Delta_k\}$ satisfies
\beq \label{general_cond}
\Delta_{k+1} \le (1 - a_k) \Delta_{k} + B_k, \ \ \ k = 1, 2, \ldots.
\eeq
Then, we have
$
\Delta_{k+1}/A_k\le (1-a_1)\Delta_1 + \sum_{i=1}^k (B_i/A_i).
$
\end{lemma}

\vgap

We are now ready to describe the main convergence properties
of the SBMD algorithm for solving nonsmooth strongly convex problems.

\begin{theorem} \label{theorem_strongly}
Suppose that \eqnok{def_strong_convexity},
\eqnok{q-growth}, and \eqnok{uniform} hold. If
\beq \label{def_theta_strong0}
\gamma_k \le \frac{bQ}{\mu}
\eeq and
\beq \label{def_theta_strong}
\theta_k = \frac{\gamma_k}{\Gamma_k} \ \ \
\mbox{with} \ \ \
\Gamma_k = \begin{cases}
1 & k = 1\\
\Gamma_{k-1} (1 - \frac{\gamma_k \mu}{bQ}) & k \ge 2,
\end{cases}
\eeq
then,
for any $N \ge 1$ and $x \in X$, we have
\beq    \label{result_strong}
\bbe [f(\bar x_N) - f(x)] \le \left(\sum_{k=1}^N \theta_k\right)^{-1}
\left[( b - \frac{\gamma_1 \mu}{ Q})\sum \limits_{i=1}^b V_i(x_1^{(i)}, x^{(i)})
+ \frac{1}{2}  \sum_{k=1}^N \gamma_k \theta_k \sum_{i=1}^b M_i^2\right].
\eeq
\end{theorem}

\begin{proof}
For simplicity, let us denote $V_i(z, x) \equiv V_i(z^{(i)}, x^{(i)})$, $g_{i_k} \equiv g^{(i_k)}(x_k)$,
and $V(z,x)=\sum \limits_{i=1}^b p_i^{-1}V_i(z, x)$.
Also let us denote $\zeta_k = (i_k, \xi_k)$ and $\zeta_{[k]} = (\zeta_1, \ldots, \zeta_k)$,
and let $\delta_k$ and $\bar \delta_k$ be defined in \eqnok{def_delta}.
By \eqnok{q-growth} and \eqnok{uniform}, we have
\beq \label{q-growth-1}
V(z,x) = b \sum_{i=1}^b V_i(\zi,\xtoi) \le
\frac{bQ}{2} \sum_{i=1}^b \|\zi-\xtoi\|_i^2 = \frac{bQ}{2} \|z - x\|^2.
\eeq
Using this observation, \eqnok{eqn_nonsmooth}, and \eqnok{def_strong_convexity}, we obtain
\begin{align}
V(x_{k+1}, x) & \le  V(x_k, x) + \gamma_k \langle g(x_k), x- x_k \rangle + \gamma_k \delta_k
+ \frac{1}{2}  \gamma _k^2 \bar \delta_k \nn\\
& \le  V(x_k, x) + \gamma_k
\left[f(x) - f(x_k) - \frac{\mu}{2} \|x- x_k\|^2 \right] + \gamma_k \delta_k
+ \frac{1}{2}  \gamma _k^2 \bar \delta_k \nn\\
& \le \left(1- \frac{\gamma_k \mu}{bQ} \right) V(x_k, x) + \gamma_k
\left[f(x) - f(x_k) \right] + \gamma_k \delta_k + \frac{1}{2}  \gamma _k^2 \bar \delta_k,\nn
\end{align}
which, in view of Lemma~\ref{tech_result_sum} (with
$a_k = 1 - \gamma_k \mu / (bQ)$ and $A_k = \Gamma_k$),
then implies that
\beq \label{strong_main_rec}
\frac{1}{\Gamma_N}V(x_{N+1}, x) \le \left(1 - \frac{\gamma_1 \mu}{b Q}\right) V(x_1,x)
+\sum_{k=1}^N \Gamma_k^{-1}\gamma_k
\left[f(x) - f(x_k) + \delta_k  + \frac{1}{2}  \gamma _k^2 \bar \delta_k \right].
\eeq
Using the fact that $V(x_{N+1},x) \ge 0$ and \eqnok{def_theta_strong}, we conclude
from the above relation that
\beq \label{result_key_strong}
\sum_{k=1}^N \theta_k [f(x_k) - f(x)] \le \left( 1 - \frac{\gamma_1 \mu}{bQ}\right)V(x_1,x)
+\sum_{k=1}^N \theta_k \delta_k
+ \frac{1}{2}  \sum_{k=1}^N \gamma_k \theta_k \bar \delta_k.
\eeq
Taking expectation on both sides of the above inequality,
and using relations \eqnok{zero-diff} and \eqnok{bnd_M},
we obtain
\[
\sum_{k=1}^N \theta_k \bbe [f(x_k) - f(x)] \le \left( 1 - \frac{\gamma_1 \mu}{b Q}\right)V(x_1,x)
+ \frac{1}{2}  \sum_{k=1}^N \gamma_k \theta_k \sum_{i=1}^b
M_i^2,
\]
which, in view of \eqnok{def_bar_x}, \eqnok{uniform} and the convexity of $f(\cdot)$, then clearly implies \eqnok{result_strong}.
\end{proof}

\vgap

Below we provide a specialized convergence result
for the SBMD method  to solve nonsmooth strongly convex problems
after properly selecting $\{\gamma_k\}$.

\begin{corollary} \label{co_stronglyconvex}
Suppose that \eqnok{def_strong_convexity}, \eqnok{q-growth}
and \eqnok{uniform} hold. If $\{\theta_k\}$ are set to \eqnok{def_theta_strong}
and $\{\gamma_k\}$ are set to
\beq \label{set_gamma_strong}
\gamma_k = \frac{2bQ}{\mu (k+1)}, \ \ k = 1, \ldots, N,
\eeq
then, for any $N \ge 1$ and $x \in X$, we have
\beq \label{result_strong1}
\bbe [f(\bar x_N) - f(x)] \le \frac{2bQ}{\mu(N+1)} \sum_{i=1}^b M_i^2.
\eeq
\end{corollary}

\begin{proof}
It can be easily seen from \eqnok{def_theta_strong} and \eqnok{set_gamma_strong}
that
\beq \label{temp_rel1}
\Gamma_k = \frac{2}{k(k+1)}, \ \
\theta_k = \frac{\gamma_k}{\Gamma_k} = \frac{b kQ}{\mu}, \ \
b - \frac{\gamma_1 \mu}{ Q}= 0,
\eeq
\beq \label{temp_rel2}
\sum_{k=1}^N \theta_k = \frac{bQ N (N+1)}{2 \mu}, \ \ \
\sum_{k=1}^N \gamma_k \theta_k
\le \frac{2 b^2Q^2 N }{ \mu^2}, \
\eeq
and
\beq \label{temp_rel3}
\sum_{k=1}^N \theta_k^2 = \frac{b^2Q^2}{\mu ^2} \frac{N (N+1) (2N+1)}{6}
\le \frac{b^2Q^2}{\mu ^2} \frac{N (N+1)^2}{3}.
\eeq
Hence, by \eqnok{result_strong},
\[
\bbe[f(\bar x_N) - f(x)] \le
 \frac{1}{2} \left(\sum_{k=1}^N \theta_k\right)^{-1}
  \sum_{k=1}^N \gamma_k \theta_k \sum_{i=1}^b
M_i^2 \le \frac{2bQ}{\mu (N+1)} \sum_{i=1}^b
M_i^2.
\]
\end{proof}

\vgap

In view of \eqnok{result_strong1},
the number of iterations performed by the SBMD method to find an $\epsilon$-solution
for nonsmooth strongly convex problems can be bound by
\[
\frac{2b}{\mu \epsilon} \sum\limits_{i=1}^b M_i^2,
\]
which is comparable to the optimal bound obtained in \cite{jn10-1,GhaLan12-2a,NO13-1}
(up to a constant factor $b$). To the best of our knowledge, no such complexity
results have been obtained before for BCD type methods in the literature.

\subsection{Large-deviation properties of SBMD for nonsmooth problems}
Our goal in this subsection is to establish the large-deviation results
associated with the SBMD algorithm under the following 
``light-tail'' assumption about the random variable $\xi$:
\beq \label{assump1.b}
\bbe\left\{ \exp \left[ \left\| G_i (x,\xi )\right\|_{i,*}^2  /M_i^2 \right]\right\} \le \exp (1), \; i=1,2,...,b.
\eeq
It can be easily seen that \eqnok{assump1.b} implies \eqnok{assump1.a}
by Jensen's inequality.
It should be pointed out that the above ``light-tail'' assumption
is alway satisified for determinisitc problems with bounded
subgradients.

For the sake of simplicity, we only consider the case when
the random variables $\{i_k\}$ in the SBMD agorithm are uniformly distributed, i.e., relation \eqnok{uniform}
holds. The following result states the large-deviation properties of
the SBMD algorithm for solving general nonsmooth problems without assuming strong convexity.

\begin{theorem} \label{the_large_dev}
Suppose that Assumptions~\eqnok{assump1.b} and \eqnok{uniform} holds.
Also assume that $X$ is bounded.
\begin{itemize}
\item [a)] For solving general nonsmooth CP problems (i.e., \eqnok{equal_theta_gamma} holds), we have
\beq \label{nonsmooth_prob}
\begin{array}{l}
\prob\left\{
f(\bar x_N) - f(x) \ge b \left(\sum\limits_{k = 1}^N \gamma _k\right)^{-1}
\left[\sum \limits_{i=1}^b V_i(x_1^{(i)}, x^{(i)}) +
\bar M^2 \sum\limits_{k=1}^N \gamma_k^2 \right. \right.\\
\left.\left. + \lambda \bar M^2 \left( \sum\limits_{k=1}^N \gamma_k^2 +  32 \, b
\sum\limits_{i=1}^b \cD_i
\sqrt{\sum\limits_{k=1}^N \gamma_k^2}\,\right) \right]
\right\} \le \exp\left(-\lambda^2/3\right) + \exp(-\lambda),
\end{array}
\eeq
for any $N \ge 1$, $x \in X$ and $\lambda > 0$,
where $\bar M = \max\limits_{i=1, \ldots, b} M_i$.
\item[b)] For solving strongly convex problems (i.e., \eqnok{def_strong_convexity},
\eqnok{q-growth}, \eqnok{def_theta_strong0}, and \eqnok{def_theta_strong} hold),
we have
\beq \label{nonsmooth_prob_strong}
\begin{array}{l}
\prob\left\{
f(\bar x_N) - f(x) \ge b \left(\sum\limits_{k = 1}^N \theta_k\right)^{-1}
\left[( b - \frac{\gamma_1 \mu}{ Q})\sum \limits_{i=1}^b V_i(x_1^{(i)}, x^{(i)}) +
{\bar M}^2 \sum\limits_{k=1}^N \gamma_k \theta_k  \right. \right.\\
\left.\left. + \lambda {\bar M}^2 \left( \sum\limits_{k=1}^N \gamma_k \theta_k +  32 b \sum\limits_{i=1}^b \cD_i
\sqrt{\sum\limits_{k=1}^N \theta_k^2}\,\right) \right]
\right\} \le \exp\left(-\lambda^2/3\right) + \exp(-\lambda),
\end{array}
\eeq
for any $N \ge 1$, $x \in X$ and $\lambda > 0$.
\end{itemize}
\end{theorem}

\begin{proof} We first show part a).
Note that by \eqnok{assump1.b}, the concavity of
$\phi(t) = \sqrt{t}$ for $t \ge 0$ and the Jensen's inequality,
we have, for any $i=1,2,...,b$,
\beq \label{assump1.b1}
\bbe \left\{\exp \left[ \left\| G_i (x,\xi )\right\|_{i,*}^2
/(2 M_i^2) \right]\right\} \le
\sqrt{\bbe\left\{ \exp \left[ \left\| G_i (x,\xi )\right\|_{i,*}^2
/M_i^2 \right]\right\}} \le
\exp (1/2).
\eeq
Also note that by \eqnok{zero-diff}, $\delta_k$, $k = 1, \ldots,N$,
is the martingale-difference. In addition, denoting
${\cal M}^2 \equiv 32 \, b^2 \bar M^2 \sum_{i=1}^b \cD_i$,
we have
\begin{align}
\bbe[\exp\left({\cal M}^{-2} \delta_k^2\right) ] &\le \sum_{i=1}^b p_i
\bbe\left[\exp\left({\cal M}^{-2} \|x - x_k\|^2 \,
\|p_i^{-1} U_i^T G_i - g(x_k)\|_*^2\right)\right]
& & \text{(by \eqnok{eqn_prob}, \eqnok{def_delta})}\nn \\
&\le \sum_{i=1}^b p_i \bbe\left\{\exp\left[2 {\cal M}^{-2} \|x - x_k\|^2
\left(b^2 \|G_i\|_*^2 + \|g(x_k)\|_*^2\right)\right] \right\}\nn
& & \text{(by definition of $U_i$ and \eqnok{uniform})}\\
&\le \sum_{i=1}^b p_i \bbe\left\{\exp\left[16 {\cal M}^{-2}
\left(\sum_{i=1}^b \cD_i\right)
\left(b^2\|G_i\|_*^2 + \sum_{i=1}^b M_i^2\right)\right] \right\} \nn
& & \text{(by \eqnok{bnd_G} and \eqnok{boundX})}\\
&\le  \sum_{i=1}^b p_i \bbe\left\{\exp \left[
\frac{b^2\|G_i\|_*^2 + \sum_{i=1}^b M_i^2}{2 b^2\bar M^2}
\right] \right\} \nn
& & \text{(by definition of ${\cal M}$)}\\
&\le \sum_{i=1}^b p_i \bbe\left\{\exp \left[
\frac{\|G_i\|_*^2}{2 M_i^2} + \frac{1}{2} \right] \right\} \le \exp(1). \nn
& & \text{(by \eqnok{assump1.b1})}
\end{align}
Therefore, by the well-known large-deviation theorem on the Martingale-difference (see, e.g., Lemma 2 of \cite{lns11}),
we have
\beq \label{nonsmooth_prob1}
\prob\left\{\sum_{k=1}^N \gamma_k \delta_k \ge \lambda {\cal M} \sqrt{\sum_{k=1}^N
\gamma_k^2} \right\} \le \exp(-\lambda^2/3).
\eeq
Also observe that under Assumption \eqnok{assump1.b},
\begin{align}
\bbe\left[\exp\left(\bar \delta_k/(b \bar M^2)\right)\right]
&\leq \sum_{i=1}^b p_i \bbe\left[\exp\left(
\left\| {G_{i } (x_k ,\xi _k )}
\right\|_{i,*}^2/\bar M^2\right)\right] \nn
& & \text{(by \eqnok{eqn_prob}, \eqnok{def_delta}, \eqnok{uniform})}\\
&\le \sum_{i=1}^b p_i \bbe\left[\exp\left(
\left\| {G_{i } (x_k ,\xi _k )}
\right\|_{i,*}^2/M_i^2\right)\right] \nn
& & \text{(by definition of $\bar M$)}\\
&\le \sum_{i=1}^b p_i \exp(1) = \exp(1).
& & \text{(by \eqnok{assump1.a})} \nn
\end{align}
Setting $\pi_k=\gamma_k^2/\sum_{k=1}^N \gamma_k^2$, we have
$
\exp\left \{\sum_{k=1}^N\pi_k \bar \delta_k / (b {\bar M}^2)\right \}\leq
\sum_{k=1}^N \pi_k \exp\{ \bar \delta_k/(b {\bar M}^2)\}.
$
Using these previous two inequalities, we have
$$
\bbe\left[\exp\left \{\sum_{k=1}^N\gamma_k^2 \bar \delta_k / (b {\bar M}^2
\sum_{k=1}^N\gamma_k^2)\right \}\right]\leq\exp\{1\}.
$$
It then follows from Markov's inequality that
\beq \label{nonsmooth_prob2}
\forall\lambda\geq0:
\Prob\left\{\sum_{k=1}^N \gamma_k^2 \bar \delta_k > (1+\lambda)
(b {\bar M}^2) \sum_{k=1}^N \gamma_k^2 \right\}\leq \exp\{-\lambda\}.
\eeq
Combining \eqnok{main_rec}, \eqnok{nonsmooth_prob1} and \eqnok{nonsmooth_prob2},
we obtain \eqnok{nonsmooth_prob}.

The probabilistic bound in \eqnok{nonsmooth_prob_strong} follows from
\eqnok{result_key_strong} and an argument similar to the one used
in the proof of \eqnok{nonsmooth_prob}, and hence the details are skipped.

\end{proof}

We now provide some specialized large-deviation results for
the SBMD algorithm with different selections of $\{\gamma_k\}$ and $\{\theta_k\}$.

\begin{corollary} \label{cor_large_dev}
Suppose that \eqnok{assump1.b} and \eqnok{uniform} hold. Also assume that $X$ is bounded.
\begin{itemize}
\item [a)] If $\{\theta_k\}$ and $\{\gamma_k\}$ are set
to \eqnok{equal_theta_gamma} and \eqnok{def_nonsmooth_prob2} for general nonsmooth problems,
then we have
\beq \label{nonsmooth_prob_cor1}
\begin{array}{l}
\prob\left\{
f(\bar x_N) - f(x) \ge \frac{b \sqrt{\sum_{i=1}^b M_i^2}}{\sqrt{2 Nb{\tilde D}^2}}
\left( 2b{\tilde D}^2+  \sum_{i=1}^b {\cal D}_i +  2 \lambda b{\tilde D}^2\right)
+ \frac{32 \lambda b^{\frac{5}{2}} \bar M^2 \sum_{i=1}^b {\cal D}_i}{\sqrt{Nb{\tilde D}^2}}
\right\} \\
\le \exp\left(-\lambda^2/3\right) + \exp(-\lambda)
\end{array}
\eeq
for any $x \in X$ and $\lambda > 0$.
\item [b)]  If $\{\theta_k\}$ and $\{\gamma_k\}$ are set
to \eqnok{def_theta_strong} and \eqnok{set_gamma_strong}
for strongly convex problems, then we have
\beq \label{nonsmooth_prob_cor2}
\begin{array}{l}
\prob\left\{
f(\bar x_N) - f(x) \ge \frac{4 (1+ \lambda) b^2 \bar M^2 Q}{(N+1) \mu }
+ \frac{64 \lambda b^2 \bar M^2 \sum_{i=1}^b {\cal D}_i}{\sqrt{3 N}}
\right\}
\le \exp\left(-\lambda^2/3\right) + \exp(-\lambda)
\end{array}
\eeq
for any $x \in X$ and $\lambda > 0$.
\end{itemize}
\end{corollary}

\begin{proof}
Note that by \eqnok{bound_V_i}, we have
$\sum \limits_{i=1}^b V_i(x_1^{(i)}, x^{(i)}) \le \sum_{i=1}^b {\cal D}_i$.
Also by \eqnok{def_nonsmooth_prob2}, we have
\[
\sum_{k=1}^N \gamma_k = \left(\frac{2 Nb{\tilde D}^2}{\sum_{i=1}^b M_i^2}\right)^\frac{1}{2}
\ \ \mbox{and} \ \
\sum_{k=1}^N \gamma_k^2 = \frac{2b{\tilde D}^2}{\sum_{i=1}^b M_i^2}.
\]
Using these identities and \eqnok{nonsmooth_prob}, we conclude that
\[
\begin{array}{l}
\prob\left\{
f(\bar x_N) - f(x) \ge b \left(\frac{\sum_{i=1}^b M_i^2}{2 Nb{\tilde D}^2}\right)^{\frac{1}{2}}
\left[\sum \limits_{i=1}^b {\cal D}_i +
2 b{\tilde D}^2\bar M^2 \left(\sum\limits_{i=1}^b M_i^2 \right)^{-1}\right. \right.\\
\left.\left. + \lambda \bar M^2 \left( 2b{\tilde D}^2  \left(\sum\limits_{i=1}^b \bar M_i^2 \right)^{-1}+  32 \sqrt{2} \, b^{\frac{3}{2}}\tilde D
\sum\limits_{i=1}^b \cD_i \left(\sum\limits_{i=1}^b \bar M_i^2 \right)^{-\frac{1}{2}}\right) \right]
\right\} \le \exp\left(-\lambda^2/3\right) + \exp(-\lambda).
\end{array}
\]
Using the fact that $\bar M^2 \le \sum_{i=1}^b M_i^2$ and simplifying the above relation,
we obtain \eqnok{nonsmooth_prob_cor1}. Similarly, relation \eqnok{nonsmooth_prob_cor2}
follows directly from \eqnok{nonsmooth_prob_strong} and a few bounds in \eqnok{temp_rel1},
\eqnok{temp_rel2} and \eqnok{temp_rel3}.
\end{proof}

\vgap

We now add a few remarks about the results obtained in Theorem~\ref{the_large_dev}
and Corollary~\ref{cor_large_dev}. Firstly, observe that by \eqnok{nonsmooth_prob1}, the number of iterations required
by the SBMD method to find an $(\epsilon, \lambda)$-solution of \eqnok{sp},
i.e., a point $\bar x \in X$ s.t. $\prob\{f(\bar x) - f^* \ge \epsilon\} \le \lambda$ can be bounded by
\[
{\cal O} \left( \frac{\log^2 (1/\lambda)}{\epsilon^2}\right)
\]
after disregarding a few constant factors.
To the best of our knowledge, now such large-deviation results have been
obtained before for the BCD methods for solving general nonsmooth CP problems,
although similar results have been established for
solving smooth problems or some composite problems~\cite{Nest10-1,Rich12-1}.

Secondly, it follows from \eqnok{nonsmooth_prob_strong} that
the number of iterations performed by the SBMD method to find an $(\epsilon, \lambda)$-solution
for nonsmooth strongly convex problems, after disregarding a few
constant factors, can be bounded by
$
{\cal O} \left( \log^2 (1/\lambda)/\epsilon^2\right),
$
which is about the same as the one obtained for solving nonsmooth problems without assuming
convexity. It should be noted, however, that this bound can be improved
to
$
{\cal O} \left( \log (1/\lambda)/\epsilon\right),
$
for example, by incorporating a domain shrinking procedure~\cite{GhaLan13-1}.

\setcounter{equation}{0}
\section{The SBMD algorithm for convex composite optimization} \label{sec_smooth}
In this section, we
present a variant of the SBMD algorithm which can make use of the smoothness properties
of the objective function of an SP problem.
More specifically, we consider convex composite optimization problems
given in the form of \eqnok{spc}, where $f(\cdot)$ is smooth and its
gradients $g(\cdot)$ satisfy
\beq \label{smooth_f}
\|  g_i(x + U_i \rho_i) - g_i (x)\|_i  \le L_i \| \rho_i\|_i \; \; \forall \; \rho_i \in \bbr^{n_i}, \ i=1,2,...,b.
\eeq
It then follows that
\beq \label{smooth_f1}
f(x+U_{i}\rho_i) \leq f(x)+\left\langle { g_{i}(x),\rho_i} \right\rangle+\frac{L_{i}}{2}\| \rho_i\|_i^2
\;\; \forall \rho_i \in \bbr^{n_i}, x \in X.
\eeq
The following assumption is made throughout this section.
\begin{assumption}\label{separation}The function $\chi (\cdot)$ is block separable, i.e., $\chi(\cdot)$ can be decomposed as
\beq\label{separable}\chi (x) = \sum \limits_{i=1}^n \chi_i(x^{(i)})\ \ \forall x \in X.
\eeq
where $\chi_i: \bbr^{n_i}\rightarrow \bbr$ are closed and convex.
\end{assumption}

\vgap

Let $V_i(\cdot , \cdot)$ defined in \eqnok{def_vi}. For a given $x \in X_i$ and $y \in R^{n_i},$
we define the composite prox-mapping as
\beq \label{def_prox_mapping_composite}
 {\cal P}_{i}(x, y, \gamma) :=\argmin_{z \in X_i} {\left \langle {y,z-x}\right \rangle+\frac{1}{\gamma}V_i(z,x)+\chi_i (x)}.
\eeq
Clearly, if $\chi(x) = 0$ for any $x \in X$, then problem \eqnok{spc} becomes
a smooth optimization problem and the composite prox-mapping \eqref{def_prox_mapping_composite} reduces
to \eqref{def_prox_mapping}.


\vgap

We are now ready to describe a variant of the SBMD algorithm
for solving smooth and composite problems.
\begin{algorithm} [H]
    \caption{A variant of SBMD for convex stochastic composite optimization}
    \label{algSBMD2}
    \begin{algorithmic}
\STATE Let $x_1 \in X$, stepsizes $\{\gamma_k\}_{k \ge 1}$, weights $\{\theta_k\}_{k \ge 1}$, and
probabilities $p_i \in [0,1]$, $i = 1, \ldots, b$, s.t. $\sum_{i=1}^b p_i = 1$
be given. Set $s_1 = 0,$ $u_i = 1$ for $i = 1, \ldots, b$, and $\theta_1=0.$

\FOR {$k=1, \ldots,N$}

\STATE 1. Generate a random variable $i_k$ according to \eqref{eqn_prob}.

\STATE 2. Update $s^{(i)}_k$, $i = 1, \ldots, b$, by \eqref{eqn_sum} and then set $u_{i_k} = k+1$.

\STATE 3. Update
$x_k^{(i)}$, $i = 1, \ldots, b$, by 
\beq \label{eqn_prox_composite}
x^{(i)}_{k+1}=
\left\{
\begin{array}{ll}
{\cal P}_{i_k}(x_{k}^{(i)}, G_{i_k}(x_{k},\xi_{k}), \gamma_k) & i = i_k,\\
x^{(i)}_{k} & i \neq i_k.
\end{array}
\right.
\eeq

\ENDFOR

{\bf Output:}  Set $s_{N+1}^{(i)} = s_{N+1}^{(i)}  + x_{N+1}^{(i)} \sum_{j = u_{i}}^{N+1} \theta_j$, $i = 1, \ldots, b$,
and  ${\bar x}_N = s_{N+1} / \sum \limits_{k=1}^{N+1} \theta_k.$
    \end{algorithmic}
\end{algorithm}

A few remarks about the above variant of SBMD algorithm for composite convex problem in place.
Firstly, similar to  Algorithm~\ref{algSBMD}, $G(x_k, \xi_k)$
is an unbiased estimator of $g(x_k)$ (i.e., \eqnok{assump1} holds).
Moreover, in order to know exactly the effect of stochastic noises in $G(x_k, \xi_k)$,
we assume that for some $\sigma_i \geq 0,$
\beq \label{assump2.a}
E[\|  G_i(x,\xi)- g_i(x)\| _{i,*}^2]\leq \sigma_i^2, \; i=1,\ldots,b.
\eeq
Clearly, if $\sigma_i = 0$, $i = 1, \ldots, b$, then
the problem is deterministic.
For notational convenience, we also denote
\beq \label{def_sigma}
\sigma :=\Big( \sum_{i=1}^b  {\sigma _i^2  }\Big)^{\frac{1}{2}}.
\eeq 

Secondly, observe that the way we compute the output $\bar x_N$ in Algorithm \ref{algSBMD2}
is slightly different from Algorithm \ref{algSBMD}.
In particular, we set $\theta_1=0$ and compute ${\bar x}_N $ of Algorithm~\ref{algSBMD2} as a weighted average of
the search points $x_2,...,x_{N+1},$ i.e.,
\beq \label{def_bar_x_composite}
\bar x_N =  \left(\sum_{k=2}^{N+1} \theta_k \right)^{-1} s_{N+1}
= \left(\sum_{k=2}^{N+1} \theta_k \right)^{-1}  \sum_{k=2}^{N+1} (\theta_k x_{k}),
\eeq
while the output of Algorithm~\ref{algSBMD} is
taken as a weighted average of $x_1,...,x_N$.

Thirdly, it can be easily seen from \eqnok{bound_V_i}, \eqnok{smooth_f}, and \eqnok{assump2.a}
that if $X$ is bounded, then
\begin{align}
\bbe[\|G_i(x,\xi)\| _{i,*}^2] &\le  2 \|g_i(x)\|_{i,*}^2 + 2 \bbe\|  G_i(x,\xi)- g_i(x)\| _{i,*}^2]
\le 2 \|g_i(x)\|_{i,*}^2 + 2 \sigma_i^2 \nn \\
&\le 2 \left[ 2 \|g_i(x) - g_i(x_1)\|_{i,*}^2  + 2 \|g_i(x_1)\|_{i,*}^2\right]
+ 2\sigma_i^2 \nn \\
&\le 4 L_i^2 \|x - x_1\|_{i,*}^2  + 4 \|g_i(x_1)\|_{i,*}^2
+ 2\sigma_i^2 \nn\\
&\le 8 L_i^2 {\cal D}_i + 4 \|g_i(x_1)\|_{i,*}^2
+ 2\sigma_i^2, \ \ i = 1, \ldots, b. \label{def_smoothM}
\end{align}
Hence, we can directly apply Algorithm~\ref{algSBMD} in the previous section  to
problem~\eqnok{spc}, and its rate of convergence is readily given by Theorem~\ref{theorem_nonsmooth}
and \ref{theorem_strongly}. However, in this section we will show that
by properly selecting $\{\theta_k\}$, $\{\gamma_k\}$, and $\{p_i\}$ in the above variant of the SBMD algorithm,
we can significantly improve the dependence of the rate of convergence of the SBMD algorithm on the
Lipschitz constants $L_i$, $i = 1, \ldots, b$.

We first discuss the main convergence properties of Algorithm~\ref{algSBMD2} for convex stochastic composite optimization without assuming strong convexity.
\begin{theorem} \label{theorem_composite}
Suppose that $\{i_k\}$ in Algorithm~\ref{algSBMD2} are uniformly distributed, i.e., \eqref{uniform} holds.
Also assume that $\{\gamma_k\}$ and $\{\theta_k\}$ are chosen such that
for any $k\ge 1$,
\begin{align}
\gamma_k \le \frac{1}{2 \bar L} \ \ \mbox{with} \ \ \bar L := \max_{i=1, \ldots, b} L_i, \label{gamma_composite} \\
\theta_{k+1}=b\gamma_k-(b-1)\gamma_{k+1}. \label{def_theta_composite}
\end{align}
Then,
under Assumption \eqref{assump1} and \eqref{assump2.a}, we have, for any $N \ge 2$,
\beq \label{main_result_smooth}
E[\phi(\bar x_N)-\phi(x^*)] \le \Big( \sum_{k=2}^{N+1} \theta_k \Big)^{-1}\left[(b-1)\gamma_1[\phi(x_1 ) -\phi(x^*)]
 + b\sum_{i=1}^bV_i(x_1 ,x^* ) +  \sigma^2 \sum_{k=1}^N \gamma_k^2 \right],
\eeq
where $x^*$ is an arbitrary solution of problem \eqnok{spc} and $\sigma$ is defined
in \eqnok{def_sigma}.
\end{theorem}

\begin{proof}
For simplicity, let us denote $V_i(z, x) \equiv V_i(z^{(i)}, x^{(i)})$, $g_{i_k} \equiv g^{(i_k)}(x_k)$,
and $V(z,x)=\sum \limits_{i=1}^b p_i^{-1}V_i(z, x)$.
Also denote $\zeta_k = (i_k, \xi_k)$ and $\zeta_{[k]} = (\zeta_1, \ldots, \zeta_k)$,
and let $\delta_{i_k}= G_{i_k}(x_k,\xi_k)- g_{i_k}(x_k)$ and $\rho_{i_k}=U_{i_k}^T(x_{k+1}-x_k)$.
By the definition of $\phi(\cdot)$ in \eqnok{spc}
and \eqnok{smooth_f1}, we have
\begin{align}
\phi (x_{k+1}) &\le  f (x_k ) + \langle {g_{i_k } (x_k ),\rho_{i_k} } \rangle
+ \frac{{L_{i_k } }}{2}\left\| {\rho_{i_k} } \right\|_{i_k}^2   + \chi (x_{k + 1} ) \nn \\
&=  f (x_k ) + \langle {G_{i_k } (x_k ),\rho_{i_k} } \rangle
+ \frac{{L_{i_k } }}{2} \left\| {\rho_{i_k} } \right\|_{i_k}^2   + \chi (x_{k + 1} )
- \langle \delta_{i_k},\rho_{i_k}  \rangle.
\end{align}
Moreover, it follows from the optimality condition of \eqnok{def_prox_mapping_composite} (see, e.g.,
Lemma 1 of \cite{Lan10-3}) and \eqnok{eqn_prox_composite} that
\begin{align}
\langle {G_{i_k } (x_k ,\xi _k ),\rho_{i_k} } \rangle
+ \chi _{i_k } (x_{k + 1}^{(i_k )} ) \le
\langle { G_{i_k } (x_k ,\xi _k ),x^{(i_k )}  - x_k^{(i_k )} } \rangle
+  \chi _{i_k } (x^{(i_k )} )  \nn \\
+ \frac{1}{{\gamma _k }}\left[ {V_{i_k} (x_k,x) - V_{i_k} (x_{k + 1},x) - V_{i_k } (x_{k + 1},x_k )} \right]. \nn
\end{align}
Combining the above two inequalities and using \eqnok{separable},
we obtain
\begin{align}
 \phi (x_{k+1}) 
&\le  f(x_k ) + \left\langle { G_{i_k } (x_k ,\xi _k ),x^{(i_k )}  - x_k^{(i_k )} } \right\rangle   +  \chi _{i_k } (x^{(i_k )} )
+ \frac{1}{{\gamma _k }}\left[ {V_{i_k} (x_k,x) - V_{i_k} (x_{k + 1},x) - V_{i_k } (x_{k + 1},x_k )} \right] \nn \\
 &  \qquad  + \frac{{L_{i_k } }}{2}\left\| {\rho_{i_k} } \right\|_{i_k}^2
 + \sum\limits_{i \ne i_k } { \chi _{i} (x_{k + 1}^{(i)} )}
 - \langle \delta_{i_k},\rho_{i_k}  \rangle . \label{rel_smooth_rec}
\end{align}
Noting that by the strong convexity of $\w_i(\cdot)$,  the Young's inequality, and \eqnok{gamma_composite}, we have
\begin{align*}
- \frac{1}{{\gamma _k }} V_{i_k } (x_{k + 1},x_k ) + \frac{{L_{i_k } }}{2}\left\| {\rho_{i_k} } \right\|_{i_k}^2 - \langle \delta_{i_k},\rho_{i_k}  \rangle
\le  - \left( {\frac{{1 }}{{2\gamma _k }} - \frac{{L_{i_k } }}{2}} \right)\left\| {\rho_{i_k} } \right\|_{i_k}^2  - \langle \delta_{i_k},\rho_{i_k}  \rangle \\
\le \frac{{\gamma _k  \left\| {\delta_{i_k} } \right\|_* ^2 }}{{2(1  - \gamma _k L_{i_k } )}}
\le \frac{{\gamma _k  \left\| {\delta_{i_k} } \right\|_* ^2 }}{{2(1  - \gamma _k \bar L )}}
\le \gamma _k  \left\| {\delta_{i_k} } \right\|_* ^2.
\end{align*}
Also observe that by the definition of $x_{k+1}$ in \eqnok{eqn_prox_composite}, \eqref{eqn_prox}, and the definition of $V(\cdot,\cdot)$,
we have $\sum\limits_{i \ne i_k } { \chi _{i} (x_{k + 1}^{(i)} )} = \sum\limits_{i \ne i_k } { \chi _{i} (x_{k}^{(i)} )}$
and ${V_{i_k } (x_k,x ) - V_{i_k } (x_{k + 1} ,x )} =[V (x_k,x ) - V (x_{k + 1} ,x )]/b$. Using these observations,
we conclude from \eqnok{rel_smooth_rec} that
 \begin{align}
   \phi (x_{k+1}) 
   &\le  f(x_k ) + \left\langle { G_{i_k } (x_k ,\xi _k ),x^{(i_k )}  - x_k^{(i_k )} } \right\rangle
   + \frac{1}{{b \gamma _k }}\left[ {V (x_k,x ) - V (x_{k + 1} ,x )} \right]  \nn \\
  & \qquad + \gamma _k  \left\| {\delta_{i_k} } \right\|_* ^2  + \sum\limits_{i \ne i_k } { \chi_{i} (x_k^{(i)} )}  +  \chi_{i_k} (x^{(i_k )} ).
  \label{rel_smooth_main}
\end{align}
Now noting that
\begin{align}
\bbe_{\zeta_{k}}\left[\left\langle { G_{i_k } (x_k ,\xi _k ),x^{(i_k )}  - x_k^{(i_k )} } \right\rangle | \zeta_{[k-1]}\right]
& = \frac{1}{b} \sum_{i=1}^b \bbe_{\xi_k} \left[\left\langle { G_{i } (x_k ,\xi _k ),x^{(i )}  - x_k^{(i )} } \right\rangle | \zeta_{[k-1]}\right] \nn\\
&= \frac{1}{b} \langle g(x_k), x - x_k \rangle \le \frac{1}{b} [f(x) - f(x_k)], \label{exp_g_smooth}\\
\bbe_{\zeta_{k}}\left[\left\| {\delta_{i_k} } \right\|_* ^2  | \zeta_{[k-1]}\right]
&= \frac{1}{b} \sum_{i=1}^b \bbe_{\xi_k} \left[\|G_i(x_k, \xi_k) - g_i(x_k)\|_{i,*}^2 | \zeta_{[k-1]}\right]
\le \frac{1}{b} \sum_{i=1}^b \sigma_i^2 = \frac{\sigma^2}{b}, \label{exp_delta_smooth} \\
\bbe_{\zeta_{k}}\left[\sum\limits_{i \ne i_k } { \chi_{i} (x_k^{(i)} )}  | \zeta_{[k-1]}\right]
&= \frac{1}{b} \sum_{j=1}^b \sum\limits_{i \ne j } { \chi_{i} (x_k^{(i)} )}
= \frac{b-1}{b} \chi (x_k), \label{exp_chi_smooth}\\
\bbe_{\zeta_{k}}\left[\chi_{i_k} (x^{(i_k )} )| \zeta_{[k-1]}\right]
&= \frac{1}{b} \sum_{i=1}^b \chi_i(x^{(i)} = \frac{1}{b} \chi(x), \label{exp_chi_i_smooth}
\end{align}
we conclude from \eqnok{rel_smooth_main} that
  \begin{align}
   \bbe_{\zeta_k}\left[\phi (x_{k+1}) + \frac{1}{{b\gamma _k } V (x_{k + 1},x )} |\zeta_{[k-1]}\right] &\le  f(x_k ) + \frac{1}{b}[f(x) - f(x_k)]
   + \frac{1}{b} \chi (x) + \frac{1}{{b\gamma _k }}\left[ {V (x_k ,x ) } \right] \nn \\
   &\qquad + \frac{\gamma_k}{b} \sigma^2 + \frac{b -1}{b} \chi(x_k) + \frac{1}{b} \chi(x) \nn\\
   &=  \frac{b-1}{b}\phi(x_k ) + \frac{1}{b}\phi(x)
   + \frac{1}{{b\gamma _k }}\left[ {V (x_k ,x ) - V (x_{k + 1},x )} \right]
   + \frac{\gamma_k}{b} \sigma^2, \nn
\end{align}
which implies that
\beq \label{recursion_composite1}
b\gamma_k \bbe[\phi (x_{k+1})- \phi (x)] + \bbe[V (x_{k + 1},x )]\le (b-1)\gamma_k \bbe[\phi(x_k ) -\phi(x)]
+ \bbe\left[ {V (x_k ,x )} \right] + \gamma_k^2 \sigma^2.
\eeq
Now, summing up the above inequalities (with $x = x^*$) for $k=1,\ldots,N$ ,
and noting that $\theta_{k+1}=b\gamma_k-(b-1)\gamma_{k+1}$,
we obtain
$$
 \sum_{k=2}^{N}\theta_{k}\bbe[\phi (x_{k})- \phi (x^*)] + b \gamma_N \bbe[\phi (x_{N+1})- \phi (x^*)] + \bbe[V (x_{N + 1},x )]
 \le (b-1)\gamma_1[\phi(x_1 ) -\phi(x^*)]   + V (x_1 ,x^* )
 + \sigma^2 \sum_{k=1}^N \gamma_k^2 ,
$$
Using the above inequality and the facts that $ V(\cdot,\cdot) \ge 0$ and $\phi(x_{N+1}) \ge \phi(x^*)$,
we conclude
\[
 \sum_{k=2}^{N+1}\theta_{k}\bbe[\phi (x_{k})- \phi (x^*)]  \le (b-1)\gamma_1[\phi(x_1 ) -\phi(x^*)]   + V (x_1 ,x^* )
 + \sigma^2 \sum_{k=1}^N \gamma_k^2 ,
\]
which, in view of \eqnok{def_sigma}, \eqnok{def_bar_x_composite} and the convexity of $\phi(\cdot)$,
clearly implies \eqnok{main_result_smooth}.
\end{proof}

\vgap

The following corollary describes a specialized convergence result of Algorithm~\ref{algSBMD2}
for solving convex stochastic composite optimization problems after
properly selecting $\{\gamma_k\}$.

\begin{corollary} \label{co_composite}
Suppose that $\{p_i\}$ in Algorithm~\ref{algSBMD2} are set to \eqref{uniform}.
Also assume that $\{\gamma_k\}$ are set to
\beq\label{def_gamma_composite}
\gamma_k=\gamma=\min \left \{ \frac{1}{2 \bar L}, \frac{\tilde D}{\sigma}\sqrt{\frac{b}{N}}\right \}
\eeq
for some $\tilde D > 0$, and $\{\theta_k\}$ are set to \eqnok{def_theta_composite}.
Then, under Assumptions \eqref{assump1} and \eqref{assump2.a}, we have
\begin{align}
\bbe\left[ {\phi (\bar x_N ) - \phi(x^* )} \right] &\le \frac{(b-1)[\phi(x_1 ) - \phi(x^* )]}{N}
+ \frac{2 b \bar L \sum_{i=1}^bV_i(x_1 ,x^* )}{N} \nn \\
& \qquad + \frac{ \sigma \sqrt{b} }{ \sqrt{N}} \left[ \frac{\sum_{i=1}^b V_i(x_1 ,x^* )}{\tilde D} + \tilde D\right].
\label{composite_key}
\end{align}
where $x^*$ is the optimal solution of problem \eqnok{spc}.
\end{corollary}

\begin{proof}
It follows from \eqnok{def_theta_composite} and \eqnok{def_gamma_composite}
that $\theta_k = \gamma_k =\gamma$, $k = 1, \ldots, N$. Using this observation and
Theorem~\ref{theorem_composite}, we obtain
$$\bbe\left[ {\phi (\bar x_N ) - \phi(x^* )} \right] \le  \frac{(b-1)[\phi(x_1 ) - \phi(x^* )]}{N} + \frac{b \sum_{i=1}^bV_i(x_1 ,x^* )}{N\gamma} + \gamma\sigma^2, $$
which, in view of \eqnok{def_gamma_composite}, then implies \eqnok{composite_key}.
\end{proof}

\vgap

We now add a few remarks about the results obtained in Corollary~\ref{co_composite}.
First, in view of \eqref{composite_key}, an optimal selection of $\tilde D$ would be
$\sqrt{\sum_{i=1}^b V_i(x_1 ,x^* )}$. In this case, \eqref{composite_key}
reduces to
\begin{align}
\bbe\left[ {\phi (\bar x_N ) - \phi(x^* )} \right] &\le \frac{(b-1)[\phi(x_1 ) - \phi(x^* )]}{N}
+ \frac{2 b \bar L \sum_{i=1}^bV_i(x_1 ,x^* )}{N}
+ \frac{ 2 \sigma \sqrt{b} \sqrt{\sum_{i=1}^b {\cal D}_i}}{ \sqrt{N}} \nn \\
&\le \frac{(b-1)[\phi(x_1 ) - \phi(x^* )]}{N}
+ \frac{2 b \bar L \sum_{i=1}^b{\cal D}_i}{N}
+ \frac{ 2 \sigma \sqrt{b} \sqrt{\sum_{i=1}^b {\cal D}_i}}{ \sqrt{N}} \label{temp_comp1}.
\end{align}
Second, if we directly apply Algorithm~\ref{algSBMD} to problem \eqnok{spc},
then, in view of \eqnok{nonsmooth_bnd_simple} and \eqnok{def_smoothM}, we have
\begin{align}
\bbe[\phi(\bar x_N ) - \phi(x^*)] &\le 2 \sqrt{ \sum \limits_{i=1}^b
\left[4 L_i^2 {\cal D}_i + 2 \|g_i(x_1)\|_{i,*}^2
+ \sigma_i^2\right] }
\frac{\sqrt{b} \sqrt{\sum_{i=1}^b {\cal D}_i}}{\sqrt{N}} \nn\\
&\le \frac{4 b \bar L \sum_{i=1}^b{\cal D}_i}{\sqrt{N}}
+ 2 \sqrt{\sum \limits_{i=1}^b  \left(2 \|g_i(x_1)\|_{i,*}^2
+ \sigma_i^2\right)}\frac{ \sqrt{b}  \sqrt{\sum_{i=1}^b {\cal D}_i}}{\sqrt{N}}.
\label{temp_comp2}
\end{align}
Clearly, the bound in \eqnok{temp_comp1} has a much weaker dependence
on the Lipschitz constant $\bar L$ than the one in \eqnok{temp_comp2}. In particular,
we can see that $\bar L$ can be as large as ${\cal O} (\sqrt{N})$ without affecting
the bound in \eqnok{temp_comp1}, after disregarding some other constant factors.


\vgap

In the remaining part of this section, we
consider the case when the objective function is strongly convex, i.e.,
the function $f(\cdot)$ in \eqref{spc} satisfies \eqref{def_strong_convexity}.
Similar to the previous section, we also assume that the prox-functions $V_i(\cdot,\cdot)$, $i = 1, \ldots, b$, satisfy
the quadratic growth condition \eqref{q-growth}.
The following theorem describes some convergence properties of the SBMD algorithm
for solving strongly convex composite problems.

\begin{theorem} \label{composite_theorem_strongly}
Suppose that \eqnok{def_strong_convexity},
\eqnok{q-growth}, and \eqnok{uniform} hold.
Also assume that the parameters $\{\gamma_k\}$ and $\{\theta_k\}$
are chosen such that for any $k \ge 1$,
\begin{align}
\gamma_k \le \min\left\{\frac{1}{2 \bar L}, \frac{bQ}{\mu}\right\}, \label{def_theta_strong_composite0}
\end{align}
\begin{align}
\theta_{k+1} =
\frac{b\gamma_k}{\Gamma_k}-\frac{(b-1)\gamma_{k+1}}{\Gamma_{k+1}} \ \ \
\mbox{with} \ \ \
\Gamma_k = \begin{cases}
1 & k = 1\\
\Gamma_{k-1} (1 - \frac{\gamma_k \mu}{bQ}) & k \ge 2.\label{def_theta_strong_composite}
\end{cases}
\end{align}
Then, for any $N \ge 2$, we have
\beq  \label{result_strong_composite}
\bbe [\phi(\bar x_N)-\phi(x^*)] \le \left[\sum_{k=2}^{N+1}\theta_{k}\right]^{-1} \left[\Big(b-\mu\gamma_1 Q \Big)\sum_{i=1}^bV_i(x_1,x^*)+(b-1)\gamma_1[\phi(x_1)-\phi(x^*)]+\sum_{k=1}^N\frac{\gamma_k^2}{\Gamma_k}\sigma^2\right],
\eeq
where $x^*$ is the optimal solution of problem \eqnok{spc}.
\end{theorem}

\begin{proof}
Observe that by the strong convexity of $f(\cdot)$, the relation in \eqnok{exp_g_smooth} can be strengthened
to
\[
\bbe_{\zeta_{k}}\left[\left\langle { G_{i_k } (x_k ,\xi _k ),x^{(i_k )}  - x_k^{(i_k )} } \right\rangle | \zeta_{[k-1]}\right]
= \frac{1}{b} \langle g(x_k), x - x_k \rangle \le \frac{1}{b} [f(x) - f(x_k) - \frac{\mu}{2} \|x - x_k\|^2].
\]
Using this observation, \eqnok{exp_delta_smooth}, \eqnok{exp_chi_smooth}, and \eqnok{exp_chi_i_smooth},
we conclude from \eqnok{rel_smooth_main} that
\begin{align*}
 \bbe_{\zeta_k}\left[\phi (x_{k+1}) + \frac{1}{{b\gamma _k }} V (x_{k + 1},x ) | \zeta_{[k-1]}\right] 
   &\le  f(x_k)+ \frac{1}{b}\Big[f(x)-f(x_k)-\frac{\mu}{2}\| x-x_k\|^2 \Big]
   + \frac{1}{{b\gamma _k }} {V (x_k ,x )} \\
   & \qquad  + \frac{\gamma_k}{b} \sigma^2 + \frac{b-1}{b} \chi(x_k) + \frac{1}{b} \chi (x)\\
&\le  \frac{b-1}{b}\phi(x_k)+\frac{1}{b}\phi(x)+ \Big(\frac{1}{b\gamma_k}-\frac{\mu }{b^2Q}\Big)V(x_k,x)
+ \frac{\gamma_k}{b} \sigma^2,
\end{align*}
where the last inequality follows from \eqref{q-growth-1}. By taking expectation w.r.t. $\xi_{[k-1]}$
on both sides of the above inequality, we conclude that, for any $ k \ge 1$,
$$
\bbe[V(x_{k+1},x^*)]\le \Big(1-\frac{\mu \gamma_k}{bQ}\Big) \bbe[ V(x_k,x^*)]+
(b-1)\gamma_k\bbe[\phi(x_k)-\phi(x^*)]-b\gamma_k \bbe \left[ {\phi(x_{k + 1} )} -\phi(x^*)\right]+  \gamma_k^2 \sigma^2,
$$
which, in view of Lemma~\ref{tech_result_sum} (with
$a_k = 1 - \gamma_k \mu /(bQ)$ and $A_k = \Gamma_k$ and $B_k=(b-1)\gamma[\phi(x_k)-\phi(x^*)]-b\gamma_kE \left[ {\phi (x_{k + 1} )} -\phi(x^*)\right]+\gamma_k^2\sigma^2$),
then implies that
\begin{align*}
\frac{1}{\Gamma_N}[V(x_{k+1},x^*)] &\le (1-\frac{\mu\gamma_1 }{bQ})V(x_1,x^*)+(b-1)\sum_{k=1}^N\frac{\gamma_k}{\Gamma_k}[\phi(x_k)-\phi(x^*)]\\
& \qquad -b\sum_{k=1}^N\frac{\gamma_k}{\Gamma_k}[\phi(x_{k+1})-\phi(x^*)]
+\sum_{k=1}^N\frac{\gamma_k^2}{\Gamma_k} \sigma^2 \\
&\le (1-\frac{\mu\gamma_1 }{bQ})V(x_1,x^*)+ (b-1)\gamma_1[\phi(x_1)-\phi(x^*)] \\
& \qquad - \sum_{k=2}^{N+1} \theta_k [\phi(x_k)-\phi(x^*)]+\sum_{k=1}^N\frac{\gamma_k^2}{\Gamma_k} \sigma^2,
\end{align*}
where the last inequality follows from \eqnok{def_theta_strong_composite}
and the fact that $\phi(x_{N+1}) - \phi(x^*) \ge 0$.
Noting that $V(x_{N+1},x^*) \ge 0,$ we conclude from the above inequality that
$$\sum_{k=2}^{N+1}\theta_{k}\bbe[\phi(x_{k})-\phi(x^*)] \le (1-\frac{\mu\gamma_1 }{bQ})V(x_1,x^*)+(b-1)\gamma_1[\phi(x_1)-\phi(x^*)]
+\sum_{k=1}^N\frac{\gamma_k^2}{\Gamma_k} \sigma^2.$$
Our result immediately follows from the above inequality,
the convexity of $\phi(\cdot)$, and \eqnok{def_bar_x_composite}.
\end{proof}

\vgap

Below we specialize the rate of convergence
of the SBMD method for solving strongly convex composite problems
with a proper selection of $\{\gamma_k\}$ .

\begin{corollary} \label{composite_stronglyconvex}
Suppose that \eqnok{def_strong_convexity}, \eqnok{q-growth},
and \eqnok{uniform} hold. Also assume that $\{\theta_k\}$ are set to \eqnok{def_theta_strong_composite} and
\beq\label{def_gamma_strong_composite}
\gamma_k = 2bQ/(\mu (k+k_0)) \ \ \forall k \ge 1,
\eeq where
$$k_0:=\left\lfloor {\frac{4bQ \bar L}{\mu}} \right\rfloor.$$
Then, for any $N \ge 2$, we have
\beq \label{result_strong1_composite}
\bbe [\phi(\bar x_N) - \phi(x^*)] \le \frac{\mu Q k_0^2}{N(N+1)}\sum_{i=1}^bV_i(x_1 ,x^* )
+ \frac{2Q (b-1) k_0 }{N (N+1)}[\phi(x_1)-\phi(x^*)] + \frac{4 b\sigma^2}{\mu Q (N+1)},
\eeq
where $x^*$ is the optimal solution of problem~\eqnok{spc}.
\end{corollary}

\begin{proof}
We can check that
\[
\gamma_k = \frac{2bQ}{\mu (k+\left\lfloor {\frac{4bQ \bar L}{\mu}} \right\rfloor)} \le \frac{1}{2 \bar L}.
\]
It can also be easily seen from the definition of $\gamma_k$ and \eqnok{def_theta_strong_composite}
that
\beq \label{temp_rel1_composite}
\Gamma_k = \frac{k_0 (k_0+1)}{(k+k_0)(k+k_0-1)}, \ \
1 - \frac{\gamma_1 \mu}{bQ} = \frac{k_0-1}{k_0+1}, \ \ \forall k \ge 1,
\eeq
\beq \label{temp_rel2_composite}
\theta_k = \frac{b\gamma_k}{\Gamma_k}-\frac{(b-1)\gamma_{k+1}}{\Gamma_{k+1}}
= \frac{2bkQ+2bQ(k_0-b)}{\mu k_0(k_0+1)} \ge \frac{2bk}{\mu Qk_0 (k_0+1)},
\eeq
and hence that
\beq \label{temp_rel3_composite}
\sum_{k=2}^{N+1} \theta_k \ge \frac{bQN (N+1)}{ \mu k_0 (k_0+1)}, \ \ \
\sum_{k=1}^N \frac{\gamma_k^2}{ \Gamma_k} = \frac{4b^2Q^2}{\mu^2 k_0(k_0+1)} \sum_{k=1}^N \frac{k+k_0-1}{k+k_0}
\le \frac{4Nb^2Q^2}{\mu^2k_0(k_0+1)} . \
\eeq
By using the above observations and \eqref{result_strong_composite}, we have
\begin{align*}
\bbe[\phi(\bar x_N)-\phi(x^*)] &\le \left(\sum_{k=2}^{N+1} \theta_k\right)^{-1}\left[\Big( 1-\frac{\mu\gamma_1}{bQ}\Big)V(x_1,x^*)+(b-1)\gamma_1[\phi(x_1)-\phi(x^*)]+\sum_{k=1}^N\frac{\gamma_k^2}{\Gamma_k}\sigma^2\right]. \cr
& \le \frac{\mu k_0(k_0+1)}{bQN(N+1)} \left[ \frac{k_0-1}{k_0+1}V(x_1,x^*)
+ \frac{2b(b-1)Q}{\mu(k_0+1)}[\phi(x_1)-\phi(x^*)] + \frac{4Nb^2Q^2 \sigma^2}{\mu^2 k_0(k_0+1)} \right] \cr
& \le \frac{\mu k_0^2}{bQN(N+1)}V(x_1,x^*) + \frac{2(b-1)k_0}{N (N+1)}[\phi(x_1)-\phi(x^*)] + \frac{4bQ \sigma^2}{\mu (N+1)},
\end{align*}
where the second inequality follows \eqref{temp_rel1_composite}, \eqref{temp_rel2_composite} and \eqref{temp_rel3_composite}.
\end{proof}

\vgap

It is interesting to observe that, in view of \eqref{result_strong1_composite} and the definition of $k_0$,
the Lipschitz constant $\bar L$ can be as large as ${\cal O}(\sqrt{N})$
without affecting the rate of convergence of the SBMD algorithm,
after disregarding other constant factors,  for solving strongly convex stochastic composite
optimization problems.

\setcounter{equation}{0}
\section{SBMD Algorithm for nonconvex composite optimization} \label{sec_nonconvex}
In this section we still consider composite optimization problems given in the form of \eqnok{spc}.
However, we assume that the smooth component $f(\cdot)$ is not necessarily convex, while
the nonsmooth component $\chi(\cdot)$ is still convex and separable (i.e., \eqnok{separable} holds).
In addition, we assume that the prox-functions satisfy the quadratic growth condition in \eqnok{q-growth}.
Our goal is to show that the SBMD algorithm, when employed with a certain randomization scheme,
can also be used to solve these nonconvex stochastic composite problems.

In order to discuss the convergence of the SBMD algorithm for
solving nonconvex composite problems, we need to first define an
appropriate termination criterion. Note that if $X = \bbr^n$ and
$\chi(x) = 0$, then a natural way to evaluate the quality of a
candidate solution $x$ will be $\|\nabla f(x)\|$. For more general
nonconvex composite problems, we introduce the notion of composite
projected gradient so as to evaluate the quality of a candidate
solution
(see~\cite{Nest04,LanMon13-1,LanMon09-1,DangLan12-1,GhaLanZhang13-1}
for some related discussions). More specifically, for a given $x
\in X$, $y \in \bbr^n$ and a constant $\gamma >0$, we define
$\PG(x, y, \gamma) \equiv (\PG_1(x, y, \gamma), \ldots, \PG_b(x, y, \gamma))$ by
\beq \label{proj_g}
\PG_i(x, y, \gamma) := \frac{1}{\gamma}[U_i^T x - \CP_i(U_i^T x, U_i^T y, \gamma)], \ \
i = 1, \ldots, b,
\eeq
where $\CP_i$ is defined in \eqnok{def_prox_mapping_composite}. In particular,
if $y = g(x)$, then we call $\PG(x, g(x), \gamma)$ the composite projected gradient
of $x$ w.r.t. $\gamma$. It can be easily seen that $\PG(x, g(x), \gamma) = g(x)$
when $X = \bbr^n$ and $\chi(x) = 0$.
Proposition~\ref{prop_PG} below relates the composite projected gradient to
the first-order optimality condition of the composite problem under a more general setting.


\begin{proposition} \label{prop_PG}
Let $x \in X$ be given and $\PG(x, y, \gamma)$ be defined as in \eqnok{proj_g}
for some $\gamma > 0$. Also let us denote $x^+ := x - \gamma \PG(x,g(x), \gamma)$.
Then there exists $p_i \in \partial \chi_i(U_i^T x^+)$ s.t.
\beq \label{app_stat_point}
 U_i^T g(x^+) + p_i \in -{\cal N}_{X_i}(U_i^T x^+) +
 {\cal B}_i\left((L_i + Q \gamma)\|\PG(x,g(x), \gamma)\|_i \right), \ \ i = 1, \ldots, b,
\eeq
where ${\cal B}_i(\epsilon) := \left\{ v \in \bbr^{n_i}: \|v\|_{i,*} \le \epsilon \right\}$
and ${\cal N}_{X_i}$ denotes the normal cone of $X_i$ at $U_i^T x$.
\end{proposition}

\begin{proof}
By the definition of $x^+$, \eqnok{def_prox_mapping_composite}, and \eqnok{proj_g}, we have
$
U_i^T x^+ = {\cal P}_i(U_i^T x,  U_i^T g(x), \gamma).
$
Using the above relation and the optimality condition of \eqnok{def_prox_mapping_composite}, we
conclude that there exists $p_i \in \partial \chi_i(U_i^T x^+)$ s.t.
\[
\langle U_i^T g(x) + \frac{1}{\gamma} \left[ \nabla \w_i(U_i^T x^+)  - \nabla \w_i(U_i^T x) \right]
+ p_i, u - U_i^T x^+ \rangle \ge 0, \ \ \ \forall u \in X_i.
\]
Now, denoting $\zeta = U_i^T [g(x) - g(x^+) + \frac{1}{\gamma} \left[ \nabla \w_i(U_i^T x^+)  - \nabla \w_i(U_i^T x^+)\right]$,
we conclude from the above relation that
$ U_i^T g(x^+) + p_i + \zeta \in -{\cal N}_{X_i}(U_i^T x^+)$.
Also noting that, by
$\|U_i^T[g(x^+) - g(x)]\|_{i,*} \le L_i \|U_i^T(x^+ - x)\|_i$
and $\|\nabla \w_i(U_i^T x^+)  - \nabla \w_i(U_i^T x^+)\|_{i,*} \le Q \|U_i^T(x^+ - x)\|_i$,
\begin{align}
\|\zeta\|_{i,*} &\le \left(L_i + \frac{Q}{\gamma}\right) \|U_i^T(x^+ - x)\|_i
= \left(L_i + \frac{Q}{\gamma}\right) \gamma \| U_i^T \PG(x,g(x),\gamma)\|_i \nn\\
&= (L_i + Q \gamma) \| U_i^T \PG(x,g(x),\gamma)\|_i. \nn 
\end{align}
Relation \eqnok{app_stat_point} then immediately follows from the above two relations.
\end{proof}

\vgap


A common practice in the gradient descent methods for solving nonconvex problems
(for the simple case when $X = \bbr^n$ and $\chi(x) = 0$)
is to choose the output solution $\bar x_N$ so that
\beq \label{best_trajectory}
\|g(\bar x_N)\|_* = \min_{k = 1, \ldots, N} \|g(x_k)\|_*,
\eeq
where $x_k$, $k = 1, \ldots, N$, is the trajectory generated by the
gradient descent method (see, e.g.,~\cite{Nest04}). However, such a procedure requires
the computation of the whole vector $g(x_k)$ at each iteration
and hence can be expensive if $n$ is large.
In this section, we address this problem by introducing
a randomization scheme into the SBMD algorithm as follows.
Instead of taking the best solution from the trajectory as in \eqnok{best_trajectory},
we randomly select $\bar x_N$ from $x_1, \ldots, x_N$ according
to a certain probability distribution. The basic scheme of
this algorithm is described as follows.

\begin{algorithm} [H]
    \caption{The Nonconvex SBMD Algorithm}
    \label{algSBMD_nonconvex}
    \begin{algorithmic}
\STATE Let $x_1 \in X$, stepsizes $\{\gamma_k\}_{k \ge 1}$ s.t. $\gamma_k < 2/L_i$, $i =1, \ldots,b$, and
probabilities $p_i \in [0,1]$, $i = 1, \ldots, b$, s.t. $\sum_{i=1}^b p_i = 1$
be given.
\FOR {$k=1, \ldots,N$}

\STATE 1. Generate a random variable $i_k$ according to \eqnok{eqn_prob}.
\STATE 2. Compute the partial (stochastic) gradient $G_{i_k}$ of $f(\cdot)$ at $x_k$ satisfying
\beq \label{variance_nonconvex}
\bbe[G_{i_k}] = U_{i_k}^T g(x_k) \ \ \mbox{and} \ \ \
\bbe[\|G_{i_k} - U_{i_k}^T g(x_k)\|_{i_k, *}] \le \bar \sigma_k^2,
\eeq
and update $x_k$ by \eqnok{eqn_prox_composite}.
\ENDFOR
\STATE Set $\bar x_N = x_R$ randomly according to
\beq \label{random_output}
\prob\left( {R = k} \right) = \frac{\gamma_k \min\limits_{i=1,\ldots,b} p_i \left(1 - \frac{L_{i}}{2} \gamma_k\right)}
{\sum_{k=1}^N \gamma_k \min\limits_{i=1,\ldots,b} p_i  \left(1 - \frac{L_{i}}{2} \gamma_k\right)}, \ \ k = 1,...,N.
\eeq
    \end{algorithmic}
\end{algorithm}

We add a few remarks about the above nonconvex SBMD algorithm.
Firstly, observe that we have not yet specified how the gradient $G_{i_k}$ is
computed. If the problem is deterministic, then we can simply set
$G_{i_k} = U_{i_k}^T g(x_k)$ and $\bar \sigma_k = 0$.
However, if the problem is stochastic, then the computation of
$G_{i_k}$ is a little complicated and we cannot simply set
$G_{i_k} = U_{i_k}^T \nabla F(x_k, \xi_k)$ (see Corollary~\ref{nonconvex_sto}).

Before establishing the convergence properties of the above nonconvex SBMD algorithm,
we will first present a technical result
which summarizes some important properties about
the composite prox-mapping and projected gradient.
Note that this result generalizes Lemma 1 and 2 in \cite{GhaLanZhang13-1}.

\begin{lemma} \label{proj_g_size}
Let $x_{k+1}$ be defined in \eqnok{eqn_prox_composite}, and denote
$\PG_k \equiv \PG_X(x_k, g(x_k), \gamma_k)$ and
$\tilde \PG_k \equiv \PG_{i_k}(x_k, U_{i_k} G_{i_k}, \gamma_k)$.
We have
\beq \label{comp_proj_ineq}
\langle G_{i_k}, \tilde \PG_k) \ge \|\tilde \PG_k\|^2 + \frac{1}{\gamma_k} \left[\chi(x_{k+1}) -\chi(x_k) \right],
\eeq
\beq \label{lip_comp_proj_result}
\|\tilde \PG_k - U_{i_k}^T \PG_k\|_{i_k} \le \|G_{i_k} - U_{i_k} g(x_k)\|_{i_k, *}.
\eeq
\end{lemma}

\begin{proof}
By the optimality condition of \eqnok{def_prox_mapping_composite} and the definition of $x_{k+1}$ in
\eqnok{eqn_prox_composite}, there exists $p \in \partial \chi_{i_k}(x_{k+1})$ such that
\beq \label{Lip_comp_proj1}
\langle G_{i_k} + \frac{1}{\gamma_k} \left[\nabla \w_{i_k}(U_{i_k}^T x_{k+1})
 - \nabla \w_{i_k}(U_{i_k}^T x_k) \right]+p , \frac{1}{\gamma_k} (u- U_{i_k}^T x_{k+1}) \rangle \ge 0, \ \ \forall u \in X_{i_k}.
\eeq
Letting $u= U_{i_k}^T x_k $ in the above inequality and re-arranging terms, we obtain
\beq \label{temp_opt}
\begin{array}{lll}
\langle G_{i_k} , \frac{1}{\gamma_k} U_{i_k}^T (x_k-x_{k+1}) \rangle
&\ge& \frac{1}{\gamma_k^2} \langle \nabla w_{i_k}(U_{i_k}^T x_{k+1}) - \nabla w_{i_k}(U_{i_k}^T x_k),
U_{i_k}^T (x_{k+1} - x_k)\rangle + \langle p ,\frac{1}{\gamma_k} U_{i_k}^T (x_k-x_{k+1})  \rangle \\
&\ge& \frac{1}{\gamma_k^2} \langle \nabla w_{i_k}(U_{i_k}^T x_{k+1}) - \nabla w_{i_k}(U_{i_k}^T x_k),
U_{i_k}^T (x_{k+1} - x_k)\rangle\\
& & \, + \frac{1}{\gamma_k} \left[\chi_{i_k}(U_{i_k}^Tx_{k+1})-\chi_{i_k}(U_{i_k}^Tx_k) \right] \\
&\ge& \frac{1}{\gamma_k^2} \|U_{i_k}^T (x_{k+1} - x_k)\|^2 + \frac{1}{\gamma_k} \left[\chi_{i_k}(U_{i_k}^Tx_{k+1})-\chi_{i_k}(U_{i_k}^Tx_k) \right] \\
&=& \frac{1}{\gamma_k^2} \|U_{i_k}^T (x_{k+1} - x_k)\|^2 +  \frac{1}{\gamma_k} [\chi(x_{k+1}) - \chi(x_k)],
\end{array}
\eeq
where the second and third inequalities, respectively, follow from the
convexity of $\chi_{i_k}$ and the strong convexity of $\w$,
and the last identity follows from the definition of $x_{k+1}$
and the separability assumption about $\chi$ in \eqnok{separable}.
The above inequality, in view of the fact that $\gamma_k \tilde \PG_k = U_{i_k}^T (x_k - x_{k+1})$
due to \eqnok{proj_g} and \eqnok{eqn_prox_composite},
then implies \eqnok{comp_proj_ineq}.

Now we show that \eqnok{lip_comp_proj_result} holds. Let us denote
$x_{k+1}^+ = x_k - \gamma_k \PG_k$. By the optimality condition of \eqnok{def_prox_mapping_composite}
and the definition of $\PG_k$, we have, for some $q \in \partial \chi_{i_k}(x_{k+1}^+)$,
\beq \label{Lip_comp_proj2}
\langle U_{i_k}^T g(x_k) + \frac{1}{\gamma_k} \left[\nabla \w_{i_k}(U_{i_k}^T x_{k+1}^+)
 - \nabla \w_{i_k}(U_{i_k}^T x_k) \right]+q , \frac{1}{\gamma_k} (u- U_{i_k}^T x_{k+1}^+) \rangle
 \ge 0, \ \ \forall u \in X_{i_k}.
\eeq
Letting $u= U_{i_k}^T x_{k+1}^+$ in \eqnok{Lip_comp_proj1} and using an argument
similar to \eqnok{temp_opt}, we have
\[
\begin{array}{lll}
\langle G_{i_k} , \frac{1}{\gamma_k} U_{i_k}^T (x_{k+1}^+-x_{k+1}) \rangle
&\ge& \frac{1}{\gamma_k^2} \langle \nabla w_{i_k}(U_{i_k}^T x_{k+1}) - \nabla w_{i_k}(U_{i_k}^T x_{k}),
U_{i_k}^T (x_{k+1} - x_{k+1}^+)\rangle \\
& & \, + \frac{1}{\gamma_k} \left[\chi_{i_k}(U_{i_k}^Tx_{k+1})-\chi_{i_k}(U_{i_k}^T x_{k+1}^+) \right].
\end{array}
\]
Similarly, letting $u= U_{i_k}^T x_{k+1}$ in \eqnok{Lip_comp_proj2}, we have
\[
\begin{array}{lll}
\langle U_{i_k}^T g(x_k) , \frac{1}{\gamma_k} U_{i_k}^T (x_{k+1}-x_{k+1}^+) \rangle
&\ge& \frac{1}{\gamma_k^2} \langle \nabla w_{i_k}(U_{i_k}^T x_{k+1}^+) - \nabla w_{i_k}(U_{i_k}^T x_{k}),
U_{i_k}^T (x_{k+1}^+ - x_{k+1})\rangle \\
& & \, + \frac{1}{\gamma_k} \left[\chi_{i_k}(U_{i_k}^T x_{k+1}^+)-\chi_{i_k}(U_{i_k}^T x_{k+1}) \right].
\end{array}
\]
Summing up the above two inequalities, we obtain
\[
\begin{array}{lll}
\langle G_{i_k} -  U_{i_k}^T g(x_k), U_{i_k}^T (x_{k+1}^+-x_{k+1}) \rangle
&\ge& \frac{1}{\gamma_k} \langle \nabla w_{i_k}(U_{i_k}^T x_{k+1}) - \nabla w_{i_k}(U_{i_k}^T x_{k+1}^+),
U_{i_k}^T (x_{k+1} - x_{k+1}^+)\rangle\\
&\ge& \frac{1}{\gamma_k} \|U_{i_k}^T (x_{k+1} - x_{k+1}^+)\|_{i_k}^2,
\end{array}
\]
which, in view of the Cauchy-Schwarz inequality, then implies that
\[
\frac{1}{\gamma_k} \|U_{i_k}^T (x_{k+1} - x_{k+1}^+)\|_{i_k} \le \| G_{i_k} -  U_{i_k}^T g(x_k)\|_{i_k,*}.
\]
Using the above relation and \eqnok{proj_g},
we have
\[
\begin{array}{lll}
\|\tilde \PG_k - U_{i_k}^T \PG_k\|_{i_k}  &=& \|\frac{1}{\gamma_k}U_{i_k}^T(x_k-x_{k+1})- \frac{1}{\gamma_k}U_{i_k}^T(x_k- x_{k+1}^+)\|_{i_k} \\
&=& \frac{1}{\gamma_k} \|U_{i_k}^T(x_{k+1}^+ - x_{k+1})\|_{i_k} \le \| G_{i_k} -  U_{i_k}^T g(x_k)\|_{i_k,*}.
\end{array}
\]
\end{proof}

We are now ready to describe the main convergence
properties of the nonconvex SBMD algorithm.

\begin{theorem} \label{theorem_nonconvex}
Let $\bar x_N = x_R$ be the output of the nonconvex SBMD algorithm.
We have
\beq \label{main_result_nonconvex}
\bbe[\|\PG_X(x_R, g(x_R), \gamma_R)\|^2] \le
\frac{\phi(x_1) - \phi^* + 2 \sum_{k=1}^N \gamma_k \bar \sigma_k^2}
{\sum_{k=1}^N \gamma_k \min\limits_{i=1,\ldots,b} p_i \left(1 - \frac{L_{i}}{2} \gamma_k\right)}
\eeq
for any $N \ge 1$, where the expectation is taken w.r.t. $i_k, G_{i_k}$, and $R$.
\end{theorem}

\begin{proof}
Denote $g_k \equiv g(x_k)$, $\delta_k \equiv G_{i_k} - U_{i_k}^T g_k$,
$\PG_k \equiv \PG_X(x_k, g_k, \gamma_k)$,
and
$\tilde \PG_k \equiv \PG_{i_k}(x_k, U_{i_k} G_{i_k}, \gamma_k)$
for any $k \ge 1$. Note that by \eqnok{eqn_prox_composite} and
\eqnok{proj_g}, we have $x_{k+1} - x_k = - \gamma_k U_{i_k} \tilde \PG_k$.
Using this observation and \eqnok{smooth_f1}, we have, for any $k = 1, \ldots, N$,
\beqa
f(x_{k+1}) &\le& f(x_k) + \langle g_k, x_{k+1}-x_k \rangle + \frac{L_{i_k}}{2} \|x_{k+1}-x_k\|^2 \label{lip_grad_f}  \nn\\
&=& f(x_k) -\gamma_k \langle g_k, U_{i_k} \tilde \PG_k \rangle + \frac{L_{i_k}}{2} \gamma_k^2 \|\tilde \PG_k \|_{i_k}^2 \nn \\
&=& f(x_k) -\gamma_k \langle G_{i_k}, \tilde \PG_k  \rangle + \frac{L_{i_k}}{2} \gamma_k^2 \|\tilde \PG_k \|_{i_k}^2
+ \gamma_k \langle \delta_k, \tilde \PG_k  \rangle.\nn
\eeqa
Using the above inequality and Lemma~\ref{proj_g_size}, we obtain
\[
f(x_{k+1}) \le f(x_k) - \left[\gamma_k \|\tilde \PG_k\|_{i_k}^2+ \chi(x_{k+1})-\chi(x_k)\right]
+ \frac{L_{i_k}}{2} \gamma_k^2 \|\tilde \PG_k\|_{i_k}^2 + \gamma_k \langle \delta_k, \tilde \PG_k  \rangle,
\]
which, in view of the fact that $\phi(x) = f(x) + \chi(x)$, then implies that
\beq \label{nonconvex_key_rel}
\phi(x_{k+1}) \le \phi(x_k) - \gamma_k \left(1 - \frac{L_{i_k}}{2} \gamma_k\right)
\|\tilde \PG_k\|_{i_k}^2 + \gamma_k \langle \delta_k, \tilde \PG_k  \rangle.
\eeq
Also observe that by \eqnok{lip_comp_proj_result}, the definition of $\tilde \PG_k$, and
the fact  $U_{i_k}^T \PG_k = \PG_{X_{i_k}}(x_k, U_{i_k}^T g_k, \gamma_k)$,
\[
\|\tilde \PG_k - U_{i_k}^T \PG_k\|_{i_k} \le \|G_{i_k} -  U_{i_k}^T g_k\|_{i_k, *} = \|\delta_k\|_{i_k, *},
\]
and hence that
\beqas
\|U_{i_k}^T \PG_k\|_{i_k}^2
&=& \|\tilde \PG_k + U_{i_k}^T \PG_k -\tilde \PG_k \|_{i_k}^2
\le 2 \|\tilde \PG_k\|_{i_k}^2 + 2\|U_{i_k}^T \PG_k -\tilde \PG_k \|_{i_k} \\
&\le& 2 \|\tilde \PG_k\|^2 + 2 \|\delta_k\|_{i_k, *}^2, \\
\langle \delta_k, \tilde \PG_k  \rangle &=& \langle \delta_k, U_{i_k}^T \PG_k \rangle
+ \langle \delta_k, \tilde \PG_k - U_{i_k}^T \PG_k \rangle
\le \langle \delta_k, U_{i_k}^T \PG_k \rangle + \|\delta_k\|_{i_k, *} \|\tilde \PG_k - U_{i_k}^T \PG_k\|_{i_k}\\
&\le& \langle \delta_k, U_{i_k}^T \PG_k \rangle + \|\delta_k\|_{i_k, *}^2.
\eeqas
By using the above two bounds and \eqnok{nonconvex_key_rel},
we obtain
\[
\phi(x_{k+1}) \le \phi(x_k) - \gamma_k \left(1 - \frac{L_{i_k}}{2} \gamma_k\right)
\left(\frac{1}{2} \|U_{i_k}^T \PG_k\|^2 - \|\delta_k\|^2 \right)+ \gamma_k \langle \delta_k,U_{i_k}^T \PG_k \rangle
+  \gamma_k \|\delta_k\|_{i_k, *}^2
\]
for any $k = 1, \ldots, N$.
Summing up the above inequalities and re-arranging the terms, we obtain
\[
\begin{array}{lll}
\sum_{k=1}^N \frac{\gamma_k}{2} \left(1 - \frac{L_{i_k}}{2} \gamma_k\right) \|U_{i_k}^T \PG_k\|^2
&\le& \phi(x_1) - \phi(x_{k+1}) + \sum_{k=1}^N \left[\gamma_k \langle \delta_k,U_{i_k}^T \PG_k \rangle
+  \gamma_k \|\delta_k\|_{i_k, *}^2\right] \\
& & + \, \sum_{k=1}^N \gamma_k \left(1 - \frac{L_{i_k}}{2} \gamma_k\right)  \|\delta_k\|_{i_k, *}^2 \\
&\le& \phi(x_1) - \phi^* + \sum_{k=1}^N \left[\gamma_k \langle \delta_k,U_{i_k}^T \PG_k \rangle
+  2 \gamma_k \|\delta_k\|_{i_k, *}^2\right],
\end{array}
\]
where the last inequality follows from the facts that $\phi(x_{k+1}) \ge \phi^*$
and $L_{i_k} \gamma_k^2 \|\delta_k\|_{i_k, *}^2 \ge 0$.
Now denoting $\zeta_k = G_{i_k}$, $\zeta_{[k]} = \{\zeta_1, \ldots, \zeta_k\}$ and
$i_{[k]} = \{i_1, \ldots, i_k\}$, taking expectation on both sides of the above inequality
w.r.t. $\zeta_{[N]}$ and $i_{[N]}$, and noting that by
\eqnok{uniform} and \eqnok{variance_nonconvex},
\beqas
\bbe_{\zeta_k}\left[\langle \delta_k,U_{i_k}^T \PG_k \rangle | i_{[k]}, \zeta_{[k-1]}\right] &=&
\bbe_{\zeta_k}\left[\langle G_{i_k} - U_{i_k}^T g_k, U_{i_k}^T \PG_k \rangle | i_{[k]}, \zeta_{[k-1]} \right] =0, \\
\bbe_{\zeta_{[N]},i_{[N]}}[ \|\delta_k\|_{i_k, *}^2]  &\le&  \bar\sigma_k^2, \\
\bbe_{i_k} \left[\left(1 - \frac{L_{i_k}}{2} \gamma_k\right)\|U_{i_k}^T \PG_k\|^2|\zeta_{[k-1], i_{[k-1]}}\right] &=&
\sum_{i=1}^b p_i \left(1- \frac{L_{i}}{2} \gamma_k\right) \|U_{i}^T \PG_k\|^2 \\
& \ge & \sum_{i=1}^b  \|U_{i}^T \PG_k\|^2 \min\limits_{i=1,\ldots,b} p_i \left(1 - \frac{L_{i}}{2} \gamma_k\right) \\
&=& \|\PG_k\|^2 \min\limits_{i=1,\ldots,b} p_i \left(1 - \frac{L_{i}}{2} \gamma_k\right),
\eeqas
we conclude that
\[
\sum_{k=1}^N  \gamma_k \min\limits_{i=1,\ldots,b} p_i \left(1 - \frac{L_{i}}{2} \gamma_k\right)
\bbe_{\xi_{[N]},i_{[N]}} \left[\|\PG_k\|^2\right]
\le \phi(x_1) - \phi^* + 2 \sum_{k=1}^N \gamma_k \bar \sigma_k^2.
\]
Dividing both sides of the above inequality by $\sum_{k=1}^N  \min\limits_{i=1,\ldots,b} p_i \left(1 - \frac{L_{i}}{2} \gamma_k\right)$,
and using the probability distribution of $R$ given in \eqnok{random_output},
we obtain \eqnok{main_result_nonconvex}.
\end{proof}

\vgap

We now discuss some consequences for Theorem~\ref{theorem_nonconvex}.
More specifically, we discuss the rate of convergence of the nonconvex
SBMD algorithm for solving deterministic and stochastic
problems, respectively, in Corollaries~\ref{nonconvex_det} and \ref{nonconvex_sto}.

\begin{corollary} \label{nonconvex_det}
Consider the deterministic case when $\bar \sigma_k = 0$, $k = 1, \ldots, N$, in \eqnok{variance_nonconvex}.
Suppose that the random variable $\{i_k\}$ are uniformly distributed (i.e., \eqnok{uniform} holds).
If $\{\gamma_k\}$ are set to
\beq \label{nonconvex_determine}
\gamma_k = \frac{1}{\bar L}, k = 1, \ldots, N,
\eeq
where $\bar L$ is defined in \eqnok{gamma_composite},
then we have, for any $N \ge 1$,
\beq \label{nonconvex_determine_result}
\bbe[\|\PG_X(x_R, g(x_R), \gamma_R)\|^2] \le \frac{2 b \bar L [\phi(x_1) - \phi^*]}{N}.
\eeq
\end{corollary}

\begin{proof}
Noting that by our assumptions about $p_i$ and \eqnok{nonconvex_determine},
we have
\beq \label{bnd_denom_nonconvex}
\min_{i=1, \ldots, b} p_i \left(1 - \frac{L_i}{2} \gamma_k\right) =
\frac{1}{b} \min_{i=1, \ldots, b} \left(1 - \frac{L_i}{2} \gamma_k\right) \ge \frac{1}{2b},
\eeq
which, in view of \eqnok{main_result_nonconvex} and the fact that $\bar \sigma_k = 0$,
then implies that, for any $N \ge 1$,
\[
\bbe[\|\PG_X(x_R, g(x_R), \gamma_R)\|^2] \le \frac{2b [\phi(x_1) - \phi^*]}{N} \frac{1}{\bar L}
= \frac{2b \bar L [\phi(x_1) - \phi^*]}{N }.
\]
\end{proof}

Now, let us consider the stochastic case when $f(\cdot)$ is given in the form of expectation (see \eqnok{sp}).
Suppose that the norms $\|\cdot\|_i$ are inner product norms in $\bbr^{n_i}$
and that
\beq \label{assump2.1}
\bbe[\|U_i \nabla F(x, \xi) - g_i(x)\|] \le \sigma \ \ \forall x \in X
\eeq
for any $i= 1, \ldots, b$. Also assume that $G_{i_k}$ is computed by
using a mini-batch approach with size $T_k$, i.e.,
\beq \label{mini_batch}
G_{i_k} = \frac{1}{T_k} \sum_{t=1}^{T_k} U_{i_k} \nabla F(x_k, \xi_{k,t}),
\eeq
for some $T_k \ge 1$, where $\xi_{k,1}, \ldots, \xi_{k,T_k}$ are i.i.d. samples
of $\xi$.

\vgap

\begin{corollary} \label{nonconvex_sto}
Assume that the random variables $\{i_k\}$ are uniformly distributed (i.e., \eqnok{uniform} holds).
Also assume that $G_{i_k}$ is computed by \eqnok{mini_batch} for $T_k = T$
and that $\{\gamma_k\}$ are set to \eqnok{nonconvex_determine}.
Then we have
\beq \label{nonconvex_stoch_result}
\bbe[\|\PG_X(x_R, g(x_R), \gamma_R, h)\|^2] \le \frac{2 b \bar L [\phi(x_1) - \phi^*] }{N}
+ \frac{4 \sum_{i=1}^b \sigma_i^2}{T}
\eeq
for any $N \ge 1$,
where $\bar L$ is defined in \eqnok{gamma_composite}.
\end{corollary}

\begin{proof}
Denote $\delta_{k,t} \equiv U_{i_k} [\nabla F(x_k, \xi_{k,t}) - g(x_k)]$
and $S_t = \sum_{i=1}^{t} \delta_{k,i}$.
Noting that $\bbe[\langle S_{t-1}, \delta_{k, t}\rangle |S_{t-1}] = 0$ for
all $t = 1, \ldots, T_k$, we have
\begin{align}
\bbe[\|S_{T_k}\|^2 ]
&= \bbe \left[ \|S_{T_k -1}\|^2 + 2 \langle S_{T_k -1},
\delta_{k,T_k} \rangle + \|\delta_{k,T_k}\|^2\right] \nn\\
&= \bbe[\|S_{T_k -1}\|^2] + \bbe [\|\delta_{k,T_k}\|^2]
= \ldots =  \sum_{t=1}^{T_k} \|\delta_{k,t}\|^2,\nn
\end{align}
which together with \eqnok{mini_batch} then imply that
the conditions in \eqnok{variance_nonconvex} hold with
$
\bar \sigma_k^2 = \sigma^2/T_k.
$
It then follows from the previous observation and \eqnok{main_result_nonconvex} that
\beqas
\bbe[\|\PG_X(x_R, g(x_R), \gamma_R, h)\|^2] \le \frac{2b [\phi(x_1) - \phi^*]}{ \frac{N}{\bar L}}
+ \frac{4 b}{N} \sum_{k=1}^N \frac{\sigma^2}{T_k}\\
\le \frac{2b \bar L [\phi(x_1) - \phi^*]}{N } + \frac{4 b \sigma^2}{T}.
\eeqas
\end{proof}

In view of Corollary~\ref{nonconvex_sto}, in order to find an $\epsilon$ solution
of problem~\eqnok{spc}, we need to have
\beq \label{settings_nonconvex}
N = {\cal O} \left (
\frac{b \bar L}{\epsilon} [\phi(x_1) - \phi^*] \right) \ \ \mbox{and} \ \ \
T = {\cal O} \left(
\frac{b \sigma^2}{\epsilon} \right),
\eeq
which implies that the total number of samples of $\xi$ required can be bounded by
\[
{\cal O} \left( b^2 \bar L \sigma^2 [\phi(x_1) - \phi^*] /\epsilon^2
\right).
\]
The previous bound is comparable, up to a constant
factor $b^2$, to those obtained in \cite{GhaLan12,GhaLanZhang13-1}
for solving nonconvex SP problems without using block decomposition.
Note that it is possible to derive and improve the
large-deviation results associated with the above complexity results,
by using a two-phase procedure similar to those in \cite{GhaLan12,GhaLanZhang13-1}.
However, the development of these results are more involved
and hence the details are skipped.

\section{Conclusions}
In this paper, we study a new class of stochastic algorithms, namely
the SBMD methods, by incorporating the block
decomposition and an incremental block averaging scheme into the
classic mirror-descent method, for solving different convex stochastic optimization
problems, including general nonsmooth, smooth, composite and strongly
convex problems. We establish the rate of convergence of these algorithms and
show that their iteration cost can be considerably smaller than that of the mirror-descent
methods. We also develop a nonconvex SBMD algorithm and establish its worst-case
complexity for solving nonconvex stochastic composite optimization problems,
by replacing the incremental block averaging scheme with a randomization scheme to compute the output
solution. While this paper focuses on stochastic optimization, 
some of our results are also new in BCD type methods
for deterministic optimization, which include the incorporation of new averaging/randomization schemes for computing
the output solution, the derivation of large-deviation results for nonsmooth optimization and
the analysis of the rate of convergence for nonsmooth strongly convex problems and
the nonconvex composite optimization problems.

\vgap

{\bf Acknowledgement:} The authors would like to thank Professors Stephen J. Wright
and Yurri Nesterov for their encouragement and inspiring discussions in the study of this topic.

\bibliographystyle{plain}
\bibliography{glan-bib}

\end{document}